\documentclass{amsart}
\usepackage{amsmath}%
\usepackage{amsfonts}%
\usepackage{amssymb}%
\usepackage{graphicx}
\usepackage{mathtools, tikz}

\newcommand{\p}{\mathcal{P}}

\newcommand{\G}{\mathcal{G}}

\newcommand{\W}{\mathcal{W}}
\newcommand{\E}{\mathcal{E}}

\newcommand{\M}{\mathcal{M}}
\newcommand{\Oo}{\mathcal{O}}

\newcommand{\Mod}{\text{Mod}}

\newcommand{\R}{\mathbb{R}}
\newcommand{\Z}{\mathbb{Z}}
\newcommand{\N}{\mathbb{N}}

\renewcommand{\S}{\mathcal{S}}

\newtheorem{theorem}{Theorem}[section]
\theoremstyle{plain}
\newtheorem{lem}[theorem]{Lemma}
\newtheorem{cla}[theorem]{Claim}
\newtheorem{cor}[theorem]{Corollary}
\newtheorem{rem}[theorem]{Remark}
\newtheorem{prop}[theorem]{Proposition}

\newtheorem{ques}{Question}

\numberwithin{equation}{subsection}
\theoremstyle{definition}
\newtheorem{defi}[theorem]{Definition}

\title{Orbits of non-simple closed curves on a surface}
\author{Jenya Sapir}

\begin{document}

\begin{abstract}
 The mapping class group of a surface $\S$ acts on the set of closed geodesics on $\S$. This action preserves self-intersection number. In this paper, we count the orbits of curves with at most $K$ self-intersections, for each $K \geq 1$. (The case when $K=0$ is already known.) We also restrict our count to those orbits that contain geodesics of length at most $L$, for each $L >0$. This result complements a recent result of Mirzakhani, which gives the asymptotic growth of the number of closed geodesics of length at most $L$ in a single mapping class group orbit. Furthermore, we develop a new, combinatorial approach to studying geodesics on surfaces, which should be of independent interest.

\end{abstract}

\maketitle

\section{Introduction}


\subsection{The main question}
Let $\S$ be a compact genus $g$ surface with $n$ boundary components, and let $X$ be a hyperbolic metric on $\S$.
Define $\G^c$ to be the set of closed geodesics on $\S$.  Let 
\[
 \G^c(L,K) = \{\gamma \in \G^c \ | \ l(\gamma) \leq L, i(\gamma, \gamma) \leq K\}
\]
be the set of closed geodesics of length at most $L$, with at most $K$ self-intersections, where $l(\gamma)$ is the length of $\gamma$ in $X$ and $i(\gamma, \gamma)$ is its geometric self-intersection number.

We are motivated by the following question about the size of $\G^c(L,K)$:
\begin{ques}
\label{ques:GeodGrowth}
 If $K = K(L)$ is a function of $L$, what can be said about the asymptotic growth of $\#\G^c(L,K)$ as $L$ goes to infinity?
\end{ques}
For a history of results related to this question, see Section \ref{sec:CountHistory}.

Mirzakhani provides an answer when $K$ is constant in \cite{Mir16}, using the following method.  Let $\Mod_\S$ be the mapping class group of $\S$. Because geodesics are unique in their free homotopy class (by, e.g. \cite[Proposition 1.4]{Primer}) we have that $\Mod_\S$ acts on $\G^c$. 
Let $\Mod_\S \cdot \gamma$ denote the orbit of $\gamma \in \G^c$. If
\[
 s(L,\gamma) = \# \{\gamma' \in \Mod_\S \cdot \gamma \ | \ l(\gamma') \leq L\}
\]
then Mirzakhani shows that
\[
 s(L,\gamma) \sim n_\gamma n_X L^{6g-6+2n}
\]
where $n_\gamma$ is a constant depending only on $\Mod_\S \cdot \gamma$, and $n_X$ is a constant depending only on the metric $X$ \cite{Mir16}. Note that we write $A(L) \sim B(L)$ if $\lim_{L \rightarrow \infty} \frac AB = 1$.

All closed geodesics in the same orbit have the same number of self-intersections. So it makes sense to consider the set
\[
 \Oo(\cdot, K) = \{\Mod_\S \cdot \gamma \ | \ \gamma \in \G^c, i(\gamma, \gamma) \leq K\}
\]
Even though there are infinitely many closed geodesics with at most $K$ self-intersections, they fall into finitely many $\Mod_\S$ orbits. So it turns out that $\Oo(\cdot, K)$ is a finite set. This fact is well known, but we include a quick explanation in Section \ref{sec:WhyFinite} for completeness.

When $K$ is constant, Mirzakhani gets asymptotic growth for $\G^c(L,K)$ by summing over all orbits in $\Oo(\cdot,K)$. In particular,
\[
 \# \G^c(L,K) \sim c_K \cdot n_X L^{6g-6+2n}
\]
where 
\[
 c_K = \sum_{\Oo(\cdot, K)} n_\gamma
\]

Knowing the asymptotic growth of $c_K$ as $K$ goes to infinity would  give a complete answer to Question \ref{ques:GeodGrowth}. This would require us to find the asymptotic growth of $\#\Oo(\cdot, K)$ as $K$ goes to infinity, but this is not yet known.

\subsection{Main results}
In this paper, we give bounds on $\#\Oo(\cdot, K)$. We get the following result as a corollary to Theorems \ref{thm:Main} and \ref{thm:LowerBound} below:
\begin{theorem}
For all $K \geq 1$,
\[
\frac{1}{12} 2^{\sqrt{\frac{ K}{12}}} \leq \#\Oo(\cdot, K) \leq e^{d_\S \sqrt K \log d_\S \sqrt K}
\]
where $d_\S$ is a constant depending only on $\S$.
\end{theorem}
This gives bounds on the constant $c_K$ above.

Note that the size of $\Oo(\cdot,0)$ is well-known. For example, if $\S$ is a closed genus $g$ surface, we have 
\[
 \# \Oo(\cdot,0) = \lfloor \frac g2 \rfloor + 1
\]
See, for example, \cite[Section 1.3]{Primer} for a complete proof, and Section \ref{sec:WhyFinite} of this paper for the main idea of why this is true.

For arbitrary $K$, we break down the set $\Oo(\cdot,K)$ further. If we choose $L$ small enough, then not all orbits $\Mod_\S \cdot \gamma$ contain a curve of length at most $L$ (see Section \ref{NotAllL}).  So we define
\[
 \Oo(L,K) = \{\Mod_\S \cdot \gamma \in \Oo(\cdot, K)\ | \ \exists \gamma' \in \Mod_\S \cdot \gamma, l(\gamma') \leq L\}
\]
That is, this is the set of orbits $\Mod_\S \cdot \gamma$ where all curves have at most $K$ self-intersections, and at least one curve has length at most $L$.

Then we bound $\#\Oo(L,K)$ as follows:
\begin{theorem}
 \label{thm:Main}
For each $L,K > 0$,
 \[
\# \Oo(L,K) \leq \min \left \{  e^{d_\S \sqrt K \log \left ( d_X \frac{L}{\sqrt K} + d_X \right )}, e^{d_\S \sqrt K \log d_\S \sqrt K} \right \}
  \]
where $d_\S$ is a constant depending only on $\S$ and $d_X$ is a constant depending only on $X$.
\end{theorem}

Moreover, we can deduce a lower bound on $\#\Oo(L,K)$ from our previous work. In our previous paper \cite{SapirL}, we give a lower bound on the size of $\G^c(L,K)$ for a pair of pants. Since the $\Mod_\S$ stabilizer of a pair of pants inside a surface $\S$ is finite, this gives us the following lower bound on $\Oo(L,K)$ for an arbitrary surface:


\begin{theorem}[Corollary of Theorem 1.1 in \cite{SapirL}]
\label{thm:LowerBound}
 Let $\S$ be any surface. If $L \geq 8 l_X$ and $K \geq 12$, we have that
\[
 \# \Oo(L,K) \geq \frac{1}{12} \min\{  2^{\frac{1}{8l_X}L},2^{\sqrt{\frac{ K}{ 12}}}\}
\]
where $l_X$ is a constant that depends only on the metric $X$.
\end{theorem}

\subsection{Other results of interest}

We develop a combinatorial model for curves on surfaces, and use this model to prove Theorem \ref{thm:Main}. Given any $\gamma \in \G^c$ and any pants decomposition $\Pi$ of the surface $\S$ we get a word $w_\Pi(\gamma)$ that can be represented by a piecewise geodesic curve on $\S$ that winds around the pants curves and takes the shortest possible path between pants curves (see Figure \ref{fig:SurfaceProjection} and Section \ref{sec:WordsAndGeodesics}.) We can approximate the length and self-intersection number of $\gamma$ from the work $w_\Pi(\gamma)$. This allows us to translate geometric questions about curves on surfaces into combinatorial problems.

This model already has applications going beyond the scope of the current paper. For example, it was used by Aougab, Gaster, Patel and the author to show that for any $\gamma \in \G^c$, there is a hyperbolic metric $X$ so that if $i(\gamma, \gamma) = K$, then $l_X(\gamma) \leq c \sqrt K$, where $c$ is a constant depending only on the surface $\S$ \cite{AGPS16}. Moreover, the injectivity radius of $X$ is bounded below by $d / \sqrt K$, where, again, $d$ depends only on $\S$. For a closed surface, both the upper bound on length and the lower bound on the injectivity radius are sharp. 

Another result of interest is that each $\gamma \in \G^c$ has a preferred pants decomposition in the following sense.
We show that for any $\gamma \in \G^c$, one may choose a pants decomposition as follows:
  \begin{prop}
 \label{prop:NicePants}
  Let $\gamma \in \G^c$ with $i(\gamma, \gamma) = K \geq 1$. Then there exists a pants decomposition $\Pi$ so that
  \[
   i(\gamma, \Pi) \leq c_\S \sqrt K
  \]
 where the constant $c_\S$ depends only on the topology of $\S$.
 \end{prop}
 Here $i(\gamma, \Pi)$ denotes the total number of intersections between $\gamma$ and each curve in $\Pi$. An interesting further question is to find an optimal bound on the constant $c_\S$. 
 
 Note that if $\Pi$ is a generic pants decomposition of $\S$, then the best we can say is that $i(\gamma, \Pi) \leq l(\gamma) l(\Pi)$, where $l(\Pi)$ is the total length of all curves in $\Pi$. Thus, the pants decompositions found in this proposition are, indeed, special.
 
\section{Background and idea of proof}

\subsection{Prior work on Question \ref{ques:GeodGrowth}}
\label{sec:CountHistory}
The problem of counting the number of closed geodesics on a hyperbolic, and more generally, a negatively curved surface, has been studied extensively. We will briefly summarize the history here. For a more detailed summary, see \cite{SapirL}.  

If $\G^c(L)$ is the set of closed geodesics of length at most $L$, then by work of Huber \cite{Huber59}, Margulis \cite{MargulisPhD}, and many others,
\[
 \# \G^c(L) \sim \frac{e^{\delta L}}{\delta L}
\]
where $\delta$ is the topological entropy of the geodesic flow on $\S$. See \cite{Margulis04} for an excellent history of this result for various types of surfaces.

Question \ref{ques:GeodGrowth} has also been studied extensively, especially for finite $K$. When $K = 0$, the fact that the number of simple closed curves grows polynomially in $L$ was shown in \cite{BS85}, and that the growth is on the order of $L^{6g-6+2n}$ was shown in \cite{Rees81, Rivin01}. The fact that $\#\G^c(L,0)$ grows asymptotically on the order of $L^{6g-6+2n}$ was shown by Mirzakhani in \cite{Mirzakhani08}. When $K = 1$, the asymptotic growth of $\#\G^c(L,K)$ was shown by Rivin in \cite{Rivin12}. For $K$ fixed, Erlandsson-Souto showed that Mirzakhani's asymptotics for $\G^c(L,K)$ on a hyperbolic metric imply the same asymptotic growth (with different constants) in any negatively curved or flat metric. (In fact, their results give these same asymptotics when the length of a curve is measured using any geodesic current.) The author has given explicit bounds on $\#\G^c(L,K)$ in terms of both $L$ and $K$ in \cite{SapirL, SapirU}.

\subsection{A complementary result on the number of orbits}
In our work, we fix a surface $\S$ and bound $\#\Oo(\cdot, K)$ on $\S$. A complementary approach is to fix $K$, and consider the set $\Oo(\cdot, K)$ on surfaces $\S_g$ of genus $g$, as $g$ goes to infinity. To distinguish the setting, let $\Oo_g(\cdot, K)$ be the set of orbits of curves with exactly $K$ self-intersections on a genus $g$ surface. In this case, Cahn, Fanoni and Petri have shown that
\[
 \#\Oo_g(\cdot, K) \sim C_K \frac{g^{K+1}}{(K+1)!}
\]
(\cite{CFP16}.) It should be noted that the dependence of the upper bound in Theorem \ref{thm:Main} on $g$ is on the order of $e^{e^g}$, so our results are far from optimal in genus. On the other hand, the constant $C_K$ in \cite{CFP16} is related to automorphisms of ribbon graphs, and is not explicit in $K$.

\subsection{Why $\#\Oo(\cdot, K)$ is finite}
\label{sec:WhyFinite}
A very coarse bound on $\#\Oo(\cdot,K)$ can be obtained from the following observation. Take any closed geodesic $\gamma$ on $\S$, so that $i(\gamma,\gamma) \leq K$. We can treat $\gamma$ as a 4-valent graph on $\S$. This graph has at most $K$ vertices and $2K$ edges. By an Euler characteristic argument, $\gamma$ has at most $K$ complimentary regions in $\S$. 

Thus, cutting $\S$ along $\gamma$ gives us at most $K$ subsurfaces with boundary. The boundary of each subsurface is composed of at most $4K$ geodesic edges. Furthermore, the topological type of each subsurface is bounded by the topology of $\S$, in the sense that the total genus and number of totally geodesic boundaries of the subsurfaces is bounded by the genus and number of boundary components of $\S$. We can recover the $\Mod_\S$ orbit of $\gamma$ by gluing these subsurfaces back together. In fact, the collection of subsurfaces, together with a gluing pattern, uniquely determines the orbit $\Mod_\S \cdot \gamma$.

There are finitely many possible collections of subsurfaces that satisfy the above conditions, namely that their topology is bounded by the topology of $\S$, and that the total number of edges in their boundary is bounded above by $4K$. The number of ways to pair off $4K$ edges is on the order of $(2K)^{2K}$. So this observation implies that $\#\Oo(L,K)$ is finite. Note that it gives an order of growth that is much worse than the one in Theorem \ref{thm:Main} for large $K$. 

On the other hand, when $K = 0$, this argument shows that 
\[
 \#\Oo(\cdot,0) = \lfloor \frac g 2 \rfloor + 1
\]
and it can be used to get an accurate count for small $K$, as well.

\subsection{The shortest curve in a $\Mod_\S$ orbit}
\label{NotAllL}
We claim above that if $L$ is small enough, then not all orbits in $\Oo(\cdot, K)$ contain curves of length at most $L$. To see why this is true, we should understand the length of the shortest curve in a $\Mod_\S$ orbit. Let $\gamma_{short}$ denote the shortest curve in $\Mod_\S \cdot \gamma$.

If $\gamma$ is simple, then $l(\gamma_{short})$ is well understood. In particular, it cannot be shorter than the injectivity radius of $X$, denoted $r_X$. On the other hand, all hyperbolic surfaces, regardless of metric, have a simple closed curve shorter than the Bers' constant, $\epsilon_b$ \cite[Chapter 5.1]{Buser}. This fact generalizes to the shortest curve in each $\Mod_\S$ orbit of simple closed curves. So if $\Mod_\S \cdot \gamma \in \Oo(\cdot, 0)$, then 
\[
 r_X \leq l(\gamma_{short}) \leq \epsilon'_b
\]
where $\epsilon'_b$ is Bers' constant for $\Mod_\S$ orbits of simple closed curves. In particular, $\Oo(L,0) = \Oo(\cdot, 0)$ whenever $L \geq \epsilon'_b$.
 
However, the length of the shortest closed curve with $K$ self-intersections grows with $K$. In particular, if $i(\gamma, \gamma) = K$, then there are constants $c_1$ and $c_2$ depending only on $X$ so that
\[
 c_1 \sqrt K \leq l(\gamma_{short}) \leq c_2 K
\]
where there are examples of curves that satisfy both the upper and lower bounds \cite{Basmajian13,Malestein,Gaster15,AGPS16}. 

Basmajian shows the lower bound, and proves that it is tight, in \cite{Basmajian13}. The fact that $l(\gamma_{short}) \leq c_2 K$ is shown in \cite{AGPS16}. Another, simpler, method to show this upper bound was communicated to us by Malestein \cite{Malestein}. It is to homotope $\gamma$ to a rose with at most $K$ petals, and apply elements of $\Mod_\S$ to make the rose as short as possible. Finally, Gaster gives a family of curves on pairs of pants whose length grows linearly with self-intersection \cite{Gaster15}, which proves that the upper bound is tight.

So if $i(\gamma, \gamma) = K$, the orbit $\Mod_\S \cdot \gamma$ will appear in $\Oo(L,K)$ for $L$ between $c_1 \sqrt K$ and $c_2 K$.

\subsection{Further questions}
The lower bound on $l(\gamma_{short})$ is generic in the following sense. Lalley shows that if we choose a curve $\gamma_L$ length at most $L$ at random among all such curves, then $i(\gamma_L, \gamma_L) \sim \kappa (l(\gamma))^2$ almost surely, as $L$ goes to infinity \cite{Lalley96}. So if we choose an orbit $\Mod_\S\cdot \gamma$ by choosing a closed geodesic $\gamma$ of length at most $L$ at random, then $l(\gamma_{short})$ would fall in the lower range of the scale that goes from $c_1 \sqrt{i(\gamma, \gamma)}$ to $c_2 i(\gamma, \gamma)$.

It would be interesting to know whether this is generic behavior if we choose at random among all orbits of curves with $K$ self-intersections, rather than among all curves of length at most $L$.
\begin{ques}
\label{ques:Limit}
 Is it true that $\displaystyle \lim_{K \rightarrow \infty} \frac{\#\Oo(\kappa \sqrt K, K)}{\#\Oo(\cdot, K)} = 1$?
\end{ques}

An affirmative answer to this question would lead to an interesting conclusion. Note that when $L = \kappa \sqrt K$, the upper and lower bounds given by Theorems \ref{thm:Main} and \ref{thm:LowerBound} give
\[
 c_1 e^{c_1 \sqrt K} \leq \#\Oo(\kappa \sqrt K, K) \leq c_2 e^{c_2 \sqrt K}
\]
for two constants $c_1, c_2$ depending on the geometry of $\S$. So an affirmative answer would indicate that the asymptotic growth of $\#\Oo(\cdot, K)$ should be of the form $e^{c \sqrt K}$, as well.

We can generalize Question \ref{ques:Limit} as follows. Given any metric $X$, the function 
\[
\ell_X : \Mod_\S \cdot \gamma \mapsto l(\gamma_{short})
\]
can be thought of as a function
\[
 \ell : \M(\S) \times \Oo \rightarrow \R
\]
where $\M(\S)$ is the moduli space of $\S$ and $\Oo$ is the set of all $\Mod_\S$ orbits of geodesics. This function sends a pair $(Y,\Mod_\S \cdot \gamma)$ to the length of the shortest curve in $\Mod_\S \cdot \gamma$ in any lift of $Y$ to Teichm\"uller space. Note that despite the choice we must make, this function is still well-defined. In \cite{AGPS16}, we show that for each orbit $\Mod_\S \cdot \gamma$, there is some point $Y(\gamma) \in \M(\S)$ where $\ell(Y, \Mod_\S \cdot \gamma) \leq c \sqrt{i(\gamma, \gamma)}$. Assume that $Y(\gamma)$ is such a point with the largest injectivity radius. Then we can ask:
\begin{ques}
 As $K$ goes to infinity, how is the set $M_K = \{Y(\gamma) \ | \ i(\gamma, \gamma) \leq K\}$ distributed in $\M(\S)$?
\end{ques}
In particular, the answer to Question \ref{ques:Limit} is positive if $M_K$ clusters in the thick part of $\M(\S)$ as $K$ goes to infinity. The answer to this question would also be interesting if it turned out that $M_K$ is distributed according to some measure on $\M(\S)$.

\subsection{Notation}
We will often use coarse bounds in the course of this paper. Our notation for them is as follows. We write
\[
 A \lesssim B \iff \exists c > 0 \text{ s.t. } A \leq c B
\]
If we say that $A \lesssim B$ and that the constant depends only on some other variable $C$, we mean that the constant $c$ above depends only on $C$. Furthermore,
\[
 A \approx B \iff A \lesssim B \text{ and } B \lesssim A
\]

\subsection{Idea of proof of Theorem \ref{thm:Main}}
\label{sec:MainIdeaOfProof}
The proof of Theorem \ref{thm:Main} can be roughly divided into the following four parts.

\subsubsection{Words corresponding to closed geodesics}
\begin{figure}[h!]
 \includegraphics{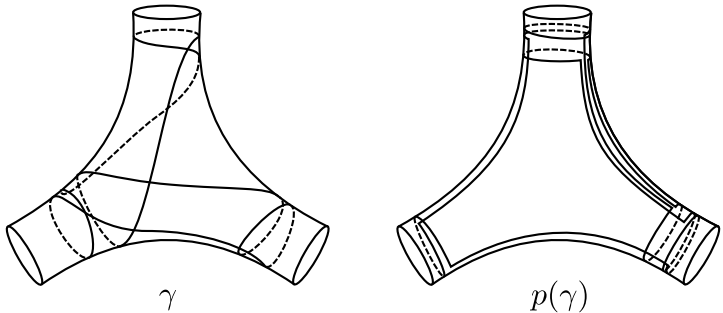}
 \caption{A curve $\gamma$ and its representative $p(\gamma)$}
 \label{fig:GammaPGamma}
\end{figure}

In Section \ref{sec:WordsAndGeodesics}, we build a combinatorial model for geodesics on $\S$ that is very similar to the one for geodesics on a pair of pants $\p$ described in 
\cite{SapirL}. In that paper, we assign a cyclic word $w(\gamma)$ to each curve $\gamma$ on $\p$ .
Specifically, the pair of pants $\p$ is decomposed into two right-angled hexagons. 
The letters of $w(\gamma)$ correspond to these (oriented) hexagon edges. Moreover, we can concatenate the hexagon edges in $w(\gamma)$ to get a path $p(\gamma)$ freely homotopic to $\gamma$ (Figure \ref{fig:GammaPGamma}).

Given a pants decomposition $\Pi$ on an arbitrary surface $\S$, we can similarly define a cyclic word $w_\Pi(\gamma)$ corresponding to each $\gamma \in \G^c$. The letters in $w_\Pi(\gamma)$ will again correspond to the edges of a hexagon decomposition coming from $\Pi$.
Concatenating the edges in $w_\Pi(\gamma)$ gives a curve freely homotopic to $\gamma$. We work with this combinatorial model for the remainder of the paper.

\subsubsection{Sets of words with special properties}
\label{Idea:SetsOfWords}
Since we only wish to count $\Mod_\S$ orbits of curves, we need to consider $\Mod_\S$ orbits of pants decompositions of $\S$. Choose a shortest representative of each $\Mod_\S$ orbit of pants decompositions. This gives us a representative list $\{\Pi_1, \dots, \Pi_l\}$ where the total length of all curves in $\Pi_i$ is the shortest in its entire $\Mod_\S$ orbit.

Given any $\Pi_i$ in the list and any length and intersection number bounds $L$ and $K$, we restrict our attention to sets of words that satisfy certain special conditions (Section \ref{sec:CountingWords}). Let $\W_i(L,K)$ be the set of words $w_{\Pi_i}(\gamma)$ so that
\begin{itemize}
 \item $|w_{\Pi_i}(\gamma)| \lesssim K$
 \item $|w_{\Pi_i}(\gamma)| \lesssim L + \sqrt K$
 \item The number of seam edges in $w_{\Pi_i}(\gamma)$ is at most $c_\S \sqrt K$, where $c_\S$ is a constant that depends only on $\S$.
\end{itemize}
where $|w_{\Pi_i}(\gamma)|$ denotes the word length of $w_{\Pi_i}(\gamma)$ and a seam edge is a hexagon edge joining two curves in $\Pi$. Note that neither of the first two conditions implies the other for all $L$ and $K$. 

\subsubsection{Each $\Mod_\S$ orbit corresponds to some $\W_i(L,K)$}
The bulk of the paper is spent in showing the following fact: For each orbit, there is some $\gamma' \in \Mod_\S \cdot \gamma$, and some pants decomposition $\Pi_i$ from the representative list so that $w_{\Pi_i}(\gamma') \in \W_i(L,K)$. First, we show that there is some $\gamma'$ and $\Pi_i$ that satisfy the two conditions that depend only on $K$ (Proposition \ref{prop:BdrySubwordIntBound}). Then we show that this choice of $\gamma'$ and $\Pi_i$ also satisfy the second condition (Lemma \ref{lem:BdrySubwordLKBounds}).

\subsubsection{Bound on $\#\Oo(L,K)$ by bounding $\#W_i(L,K)$}
The above fact allows us to define an injective map 
\[
  \Oo(L,K) \rightarrow \bigcup_{i=1}^l W_i(L,K)
\]
Thus, we can bound $\#\Oo(L,K)$ from above by bounding the size of $\W_i(L,K)$ for each $i = 1, \dots, l$. To do this, all we have to do is count the number of words $w_\Pi(\gamma)$ that satisfy the three conditions above (Section \ref{sec:CountingWords}.)

\subsection{Acknowledgements}
This work is an extension of some results from the author's thesis. She would like to thank her advisor, Maryam Mirzakhani, for the discussions and support that made this paper possible. The author would also like to thank Kasra Rafi and Ser-Wei Fu for the enlightening discussions that contributed significantly to key portions of this paper.

\section{Words and geodesics}
\label{sec:WordsAndGeodesics}

Consider an arbitrary surface $\S$. Given a pants decomposition $\Pi$ of $\S$, and any $\gamma \in \G^c$, our goal is to define a cyclic word $w_\Pi(\gamma)$ (Figure \ref{fig:SurfaceProjection}). 

Let $\Pi = \{\beta_1, \dots, \beta_m\}$ be a pants decomposition of $\S$. Here, $\beta_1, \dots, \beta_m$ are simple closed curves that cut $\S$ into pairs of pants, and include the curves in the boundary of $\S$. 

Cut each pair of pants given by $\Pi$ into two hexagons using seam edges so that 
\begin{itemize} 
\item The seam edges match up across curves in $\Pi$. That is, the hexagon decomposition looks like a 4-valent graph on the interior of $\S$.
\item Each curve in $\Pi$ is cut into two congruent arcs.
 \item The total length of all the seam edges is as small as possible. 
\end{itemize}
(See Figure \ref{fig:HexagonDecomposition} for an example.)
\begin{figure}[h!]
 \centering 
 \includegraphics{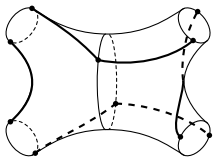}
 \caption{Hexagon decomposition of a four-holed sphere.}
 \label{fig:HexagonDecomposition}
\end{figure}

 Ideally, we would cut each pair of pants into right-angled hexagons whose corners match up. But this is impossible for most hyperbolic metrics. So we force the corners of our hexagons to match up, and then ask that they be as close to right-angled hexagons as possible.

\begin{figure}[h!]
 \centering 
 \includegraphics{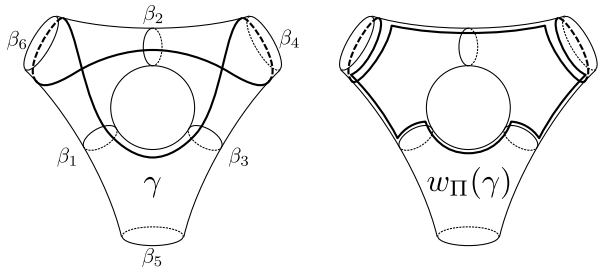}
 \caption{The word $w_\Pi(\gamma)$ corresponding to the curve $\gamma$ and pants decomposition $\Pi = \{\beta_1, \dots, \beta_6\}$.}
 \label{fig:SurfaceProjection}
\end{figure}

Let $\E = \E(\Pi)$ be the set of oriented hexagon edges for this hexagonal decomposition of $\S$. (So we will have two copies of each edge, one for each orientation.) Edges that lie on the curves $\beta_1, \dots, \beta_m$ are again called \textbf{boundary edges}, and edges that join curves in $\Pi$ together are called \textbf{seam edges}. 

We will define a cyclic word $w_\Pi(\gamma)$, whose letters will correspond to edges in $\E$. It turns out that $w_\Pi(\gamma)$ will belong to a set of allowable words, defined as follows.

\begin{defi}
\label{defi:AllowableWords_Surface}
 Fix a pants decomposition $\Pi$ and let $\E$ be the edge set for its hexagon decomposition. Let $\W_\Pi$ be the set of all cyclic words $w$ in edges in $\E$ so that
 \begin{itemize}
  \item We can write 
  \[
  w = b_1 s_1 \dots b_n s_n
  \]
  where $b_i$ is a sequence of boundary edges, $s_i$ is a single seam edge, unless $n = 1$, in which case $w = b_1$ and $s_1 = \emptyset$.
  \item The edges in $w$ can be concatenated into a closed path $p$ that does not back-track.
  \item If the subword $s_{i-1} b_i s_i$ lies in a single pair of pants for some $i$, then $|b_i| \geq 2$.
  \item No more than three consecutive edges lie on the boundary of the same hexagon.
 \end{itemize}

\end{defi}

 When we say that a subword $s_{i-1}b_i s_i$ lies in a single pair of pants, we mean that the arc corresponding to it does not cross any curve in $\Pi$. In this case, we call $b_i$ an \textbf{interior} boundary subword. 
  
  Note that if the subword $s_{i-1} b_i s_i$ does not lie in a single pair of pants, we allow $b_i$ to have any length. In particular, it can be empty. When $s_{i-1} b_i s_i$ does not lie in a single pair of pants, we call $b_i$ a \textbf{bridging} boundary subword. (See Figure \ref{fig:BoundarySubarcs} for examples of interior and bridging boundary subwords.)

\begin{lem}
\label{lem:SurfaceProjection}
 For each $\gamma \in \G^c$ there is a cyclic word $w = w_\Pi(\gamma) \in \W$ whose edges can be concatenated into a path $p$ freely homotopic to $\gamma$ (Figure \ref{fig:SurfaceProjection}.)
\end{lem}
\begin{proof}
 The construction of words for an arbitrary surface will be almost identical to the construction for pairs of pants found in \cite[Lemma 2.2]{SapirL}, but we will summarize it here for completeness.
 \begin{figure}[h!]
  \centering
  \includegraphics{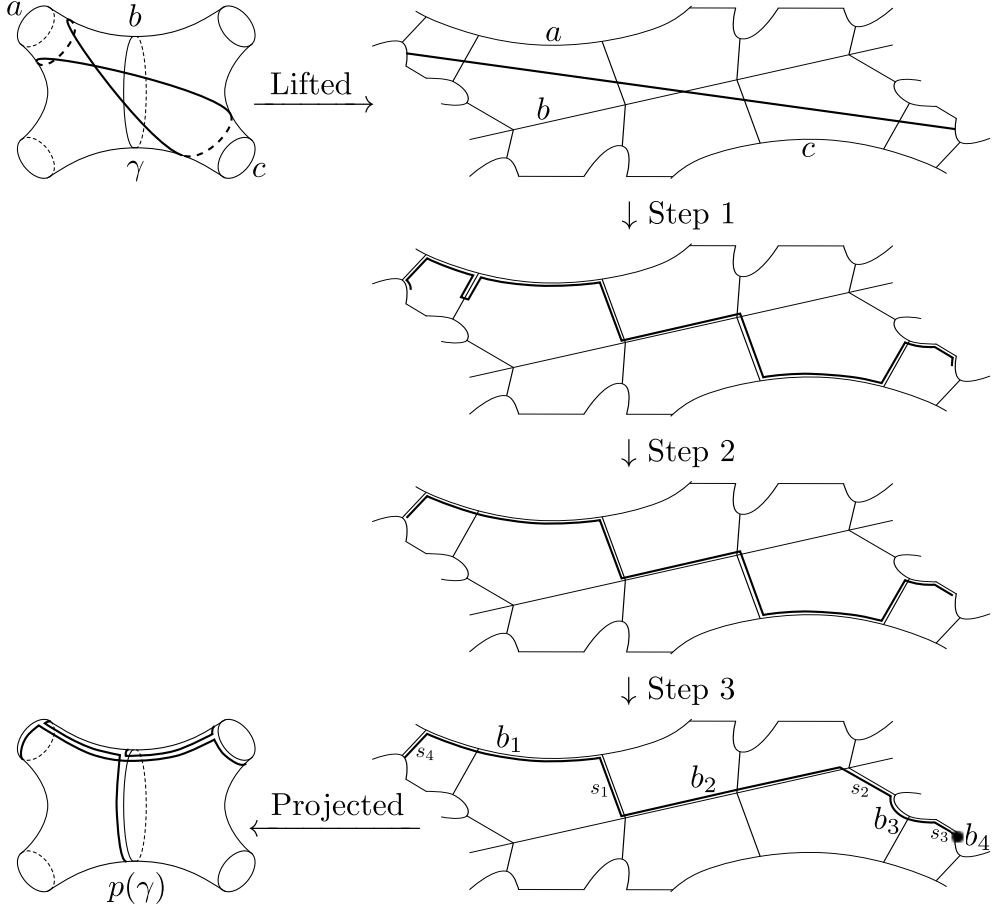}
  \caption{A closed geodesic $\gamma$ gets assigned a curve $p(\gamma) = b_1 s_1 \dots b_4 s_4$. Note that $b_4$ is empty, and is represented by a dot.}
  \label{fig:CurveProjection}
 \end{figure}

 We first construct a closed curve $p(\gamma)$ that will be the concatenation of edges in $\E$ and freely homotopic to $\gamma$. We get the cyclic word $w_\Pi(\gamma)$ by reading off the edges in $p(\gamma)$. We then show that $w_\Pi(\gamma)$ lies in the set $\W$ of allowable words.
  
  The hexagon decomposition of $\S$ cuts $\gamma$ into \textbf{segments}, which are maximal subarcs of $\gamma$ that lie entirely in a single hexagon.
  
  \textbf{Step 1.} Suppose we have a segment $\sigma$ of $\gamma$ that lies in some hexagon $h$. We replace it by an arc $p(\sigma)$ with the same endpoints, that lies in the boundary of $h$. There are two choices of $p(\sigma)$. We choose the one that passes through the fewest number of full hexagon edges.
  
  The arc $p(\sigma)$ is obtained from $\sigma$ via a homotopy relative its endpoints. When we homotope all segments of $\gamma$ in this way, we get a closed curve $p'(\gamma)$ that is freely homotopic to $\gamma$.

  \textbf{Step 2.} Since no two consecutive segments lie in the same hexagon, $p'(\gamma)$ can only back-track along seam edges. We homotope away all back-tracking and get a new closed curve $p''(\gamma)$. The curve $p''(\gamma)$ is now a concatenation of edges in $\E$. For more details, see \cite[Move 2, Lemma 2.2]{SapirL}.

  \textbf{Step 3.}  We read off the edges in $p''(\gamma)$ to get a cyclic word $w = b_1s_1 \dots b_ns_n$ where $b_i$ is a (possibly empty) sequence of boundary edges called a boundary subword and $s_i$ is a single seam edge. However, we may need to homotope $p''(\gamma)$ further so that $|b_i| \leq 1$ implies $b_i$ is a bridging boundary subword.
  
  Suppose $b_i$ is empty for some $i$. Then $s_{i-1} s_i$ is a subarc of $p''(\gamma)$. Because $p''(\gamma)$ does not back-track, $s_{i-1}$ and $s_i$ must come in to some curve $\beta_j$ in $\Pi$ from opposite sides. Thus, the curve $s_{i-1}s_i$ crosses a curve in $\Pi$. Therefore, $b_i$ is a bridging subword.
  
  Now suppose $|b_i| = 1$ for some $i$, but $b_i$ is an interior boundary subword. Then $s_{i-1} b_i s_i$ lies in a single hexagon. We can homotope $s_{i-1} b_i s_i$ to a path $x_{i-1} s x_{i+1}$ on the other side of the hexagon, where $x_{i-1}$ and $x_{i+1}$ are boundary edges, and $s$ is a seam edge (Figure \ref{fig:Move3}). Then the boundary subword $b_{i-1}$ of $p''(\gamma)$ becomes $b_{i-1} x_{i-1}$. In other words, it becomes longer. Thus, if $b_{i-1}$ is an interior boundary subword, then $b_{i-1}x_{i-1}$ is also an interior boundary subword of word length at least two. The same is true for $b_{i+1}$, which becomes $x_{i+1} b_{i+1}$. So after this homotopy, we have at least one fewer interior boundary subword with length at most 1. 
  \begin{figure}[h!]
   \centering 
   \includegraphics{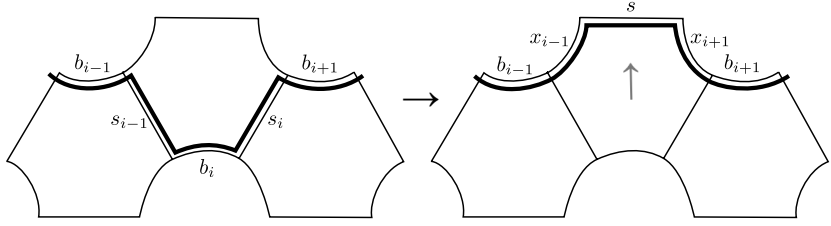}
   \caption{Step 3: we homotope away interior boundary subwords of length 1.}
   \label{fig:Move3}
  \end{figure}

  We get rid of all interior boundary subwords of length at most 1 in this way to get a curve $p(\gamma)$.
Thus, we get a cyclic word $w_\Pi(\gamma) \in \W$.

\end{proof}

\subsection{Multiple choices for $w_\Pi(\gamma)$}
  Note that the word $w_\Pi(\gamma)$ obtained in Lemma \ref{lem:SurfaceProjection} depends on the order in which we do the homotopies in Step 3 of the proof. Thus there may be more than one choice of cyclic word corresponding to each $\gamma \in \G^c$. So among all cyclic words $w$ corresponding to $\gamma$, we choose one with the fewest number of boundary subwords. There there is more than one such word, then we choose one at random.

 \subsection{Pants decompositions not on the representative list}
   \label{sec:PantsNotOnTheList}
   Recall from Section \ref{sec:MainIdeaOfProof} that the proof of Theorem \ref{thm:Main} relies on two key lemmas. Proposition \ref{prop:BdrySubwordIntBound} states that any $\gamma \in \G^c(L,K)$ has a curve $\gamma' \in \Mod_\S \cdot \gamma$ in its orbit, so that for some pants decomposition $\Pi_i$ from a representative list of pants decompositions, 
   \[
    |w_{\Pi_i}(\gamma')| \lesssim K
   \]
  and if $w_{\Pi_i}(\gamma')$ has $n$ seam edges, then $n \lesssim \sqrt K$. Lemma \ref{lem:BdrySubwordLKBounds} states that the pair $(\gamma', \Pi_i)$ satisfy a length bound, as well:
  \[
    |w_{\Pi_i}(\gamma')| \lesssim L + \sqrt K
  \]
  
  Suppose that the above two conditions from  Proposition \ref{prop:BdrySubwordIntBound} and Lemma \ref{lem:BdrySubwordLKBounds} hold for $\gamma' \in \Mod_\S \cdot \gamma$ and some pants decomposition $\Pi$, not necessarily on the representative list. Then $\Pi$ is in the same $\Mod_\S$ orbit as some $\Pi_i$ from the representative list of pants decompositions. Thus there is some homeomorphism $f$ of $\S$ sending the hexagon decomposition of $\Pi$ to the hexagon decomposition of $\Pi_i$. In particular, $f$ sends boundary edges to boundary edges and seam edges to seam edges. Thus, $f$ induces a map of words, with
   \[
    w_\Pi(\gamma) \mapsto f^*w_\Pi(\gamma)
   \]
   Then note that
   \[
    f^*w_\Pi(\gamma) = w_{f \cdot \Pi}(f \cdot \gamma)
   \]
   where $f \cdot \Pi = \Pi_i$. Therefore, if we can show the intersection number and length conditions for \textit{some} pants decomposition $\Pi$ and \textit{some} $\gamma' \in \Mod_\S \cdot \gamma$, then we have shown them for a pants decomposition from the representative list and (another) curve in the orbit of $\gamma$.


\section{Intersection number condition}
\label{sec:BoundingIntNumber}
In Section \ref{Idea:SetsOfWords}, we defined a representative list $\{\Pi_1, \dots, \Pi_l\}$ of pants decompositions, containing the shortest representative of each $\Mod_\S$ orbit. Then $\W_i(L,K)$ was the set of words $w_{\Pi_i}(\gamma) = b_1s_1 \dots b_n s_n$ so that, in particular, $|w_{\Pi_i}(\gamma)| \lesssim K$ and $n \lesssim \sqrt K$. We now show that each orbit $\Mod_\S \cdot \gamma$ contains a curve $\gamma'$ so that these two conditions are satisfied for some $\Pi_i$ in the representative list.

\begin{prop}
\label{prop:BdrySubwordIntBound}
 Let $L,K > 0$. For each $\gamma \in \G^c(L,K)$ there is a pants decomposition $\Pi_j \in \{\Pi_1, \dots, \Pi_l\}$ and a geodesic $\gamma' \in \Mod_\S \cdot \gamma$ so that if $w_{\Pi_j}(\gamma') = b_1s_1 \dots b_n s_n$ then
 \[
  \sum_{i=1}^n |b_i| \lesssim K
 \]
 and furthermore,
 \[
  n \lesssim \sqrt K
 \]
where the constants depend only on the topology of $\S$.
\end{prop}
Thus, both the total length of all boundary subwords, and the number of boundary subwords, are bounded in terms of intersection number.

\subsection{Idea of proof of Proposition \ref{prop:BdrySubwordIntBound}}
For each $\gamma \in \G^c$ and pair of pants $\Pi$, we want to relate $i(\gamma, \gamma)$ to the lengths of boundary subwords in $w_\Pi(\gamma)$. We do this as follows.
\begin{itemize}
 \item We wish to assign self-intersection points of $\gamma$ to pairs of boundary subwords in $w_\Pi(\gamma)$. In particular, each boundary subword $b_i$ of $w_\Pi(\gamma)$ corresponds to a so-called boundary subarc $\gamma_i \subset \gamma$ (Figure \ref{fig:BoundarySubarcs} and Definition \ref{def:TwistTie}). Pairs of boundary subwords are assigned the intersection points between their corresponding boundary subarcs. 
 \item We first relate $i(\gamma, \gamma)$ to the intersection numbers of all of the boundary subarcs. This is non-trivial because boundary subarcs do not partition the curve $\gamma$. We show
 \[
  \sum_{i, j = 1}^n i(\gamma_i, \gamma_j) \lesssim i(\gamma, \gamma) 
 \]
 (Lemma \ref{lem:OverlapBound}.)
 \item We then relate the intersections between a pair $(\gamma_i,\gamma_j)$ of boundary subarcs with the word lengths $|b_i|$ and $|b_j|$ of the corresponding boundary subwords. The relationships depend on whether $b_i$ and $b_j$ are interior or bridging boundary subwords.
 
 So we prove Proposition \ref{prop:BdrySubwordIntBound} for the two types of subwords separately: Lemma \ref{lem:BridgingBdrySubwordIntBound} proves it for bridging boundary subwords and Lemma \ref{lem:InteriorBdrySubwordIntBound}, for interior boundary subwords. In other words, we show that the number of bridging (resp. interior) boundary subwords is at most $c\sqrt K$, and that the total length of all bridging (resp. interior) boundary subwords is at most $c K$ for some constant $c$ depending only on $\S$. 

  \item Section \ref{sec:ProofBdrySubwordIntBound} has the proof of Proposition \ref{prop:BdrySubwordIntBound} given Lemmas \ref{lem:BridgingBdrySubwordIntBound} and \ref{lem:InteriorBdrySubwordIntBound}.
\end{itemize}


\subsection{Boundary subarcs}
\label{sec:TwistTies}

 \begin{figure}[h!]
 \centering
  \includegraphics{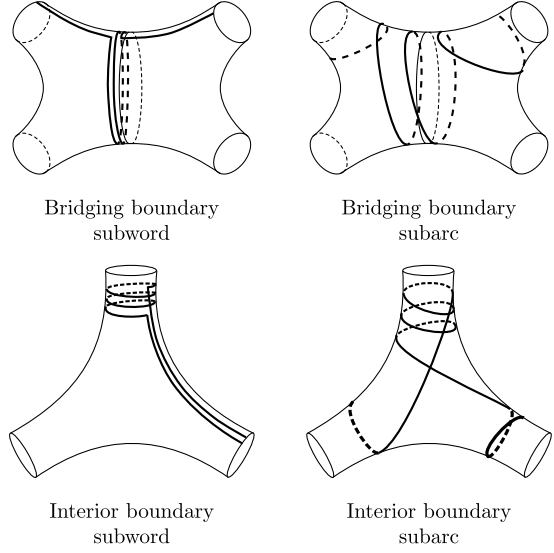}
 \caption[A boundary subarc]{The boundary subarc $\gamma_i$ corresponding to boundary subword $b_i$.}
 \label{fig:BoundarySubarcs}
 \end{figure} 

 Let us formally define boundary subarcs. Take a pants decomposition $\Pi$ and a non-simple closed curve $\gamma \in \G^c$. Note that we do not need to worry about simple closed curves, since we know $\#\Oo(L,0) = 1 + \lfloor \frac g 2 \rfloor$.

Let $b_i$ be a boundary subword of $w_\Pi(\gamma)$. To define the boundary subarc $\gamma_i$ associated to $b_i$, we work in the universal cover $\tilde \S$ of $\S$.  For an illustration of what follows, see Figure \ref{fig:BoundaryArc}. 

Note that the hexagon decomposition of $\S$ associated to $\Pi$ lifts to a hexagonal tiling of $\tilde \S$.  

\begin{defi}
\label{def:HexagonNeighborhood}
 Let $\tilde \beta$ be the lift of a curve in $\Pi$ to $\tilde \S$ . The \textbf{one-hexagon neighborhood} of $\tilde \beta$, denoted $N_1(\tilde \beta)$, is the set of all hexagons adjacent to $\tilde \beta$. Inductively, its \textbf{$k$-hexagon neighborhood} $N_k(\tilde \beta)$ is the set of all hexagons that share an edge with its $k-1$-hexagon neighborhood $N_{k-1}(\tilde \beta)$ (Figure \ref{fig:TwoHexNbhd}.)
\end{defi}

 \begin{figure}[h!] \centering
 \includegraphics{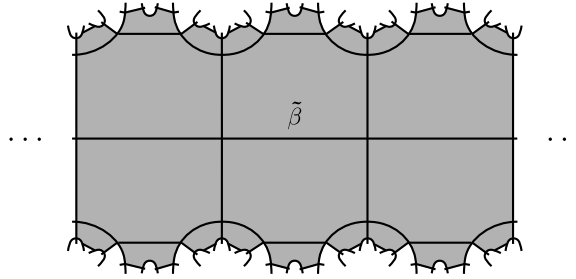}
  \caption[A two-hexagon neighborhood]{$N_2(\tilde \beta)$}
  \label{fig:TwoHexNbhd}
 \end{figure}

Lift $\gamma$ to a geodesic $\tilde \gamma$ in $\tilde \S$ . Let $p(\gamma)$ be the closed curve formed by concatenating the edges in $w_\Pi(\gamma)$. Lifting the homotopy between $\gamma$ and $p(\gamma)$ determines a lift $\tilde p(\gamma)$. Then each boundary subword $b_i$ lifts to a family of subarcs of $\tilde p(\gamma)$. 

Since $\gamma$ is non-simple, we know that $w_\Pi(\gamma)$ must have at least two boundary subwords. (If $w_\Pi(\gamma)$ has just one boundary subword, then $w_\Pi(\gamma) = b_1$. If $b_1$ lies on boundary component $\beta$ of $\Pi$, then $\gamma$ must be a power of $\beta$, which is simple.) Let $\pi : \tilde \S \rightarrow \S$  be the projection back down to $\S$.

We are now ready to define boundary subarcs associated to each boundary subword.

\begin{defi}
\label{def:TwistTie}
 Let $w_\Pi(\gamma) = b_1 s_1 \dots b_n s_n$ with $n \geq 2$. For each $i$, we define the \textbf{boundary subarc} $\gamma_i = \alpha(b_i)$ for the pair $(\gamma,\Pi)$ as follows. 
 
 Let $\tilde b_i \subset \tilde p(\gamma)$ be a lift of the boundary subword $b_i$. Then $\tilde b_i$ lies on a complete geodesic $\tilde \beta_i$. Let
 \[
  \tilde \gamma_i = 
  \left \{ 
  \begin{array}{ccc}
   \tilde \gamma \cap N_1(\tilde \beta_i) & \mbox{ if } & 2 \leq n \leq 5 \\
   \tilde \gamma \cap N_2(\tilde \beta_i) & \mbox{ if } & n \geq 6
  \end{array}
\right .
 \]
 (If $ \tilde \gamma_i$ is disconnected, take the connected component that passes through $N_1(\tilde \beta)$.) Then the boundary subarc $\gamma_i$ is defined by
\[
 \gamma_i = \pi (\tilde \gamma_i)
\]
\end{defi}
For an illustration of boundary subarcs lifted to the universal cover, see Figure \ref{fig:BoundaryArc}. For boundary subarcs on the surface, see Figure \ref{fig:BoundarySubarcs}.
 \begin{rem}
  Lemma \ref{lem:RelevantAndBridgingPercent} is the only place where it is important to use two-hexagon neighborhoods in this definition. Everywhere else, a one-hexagon neighborhood would suffice.
 \end{rem}

 \begin{figure}[h!] \centering
  \includegraphics{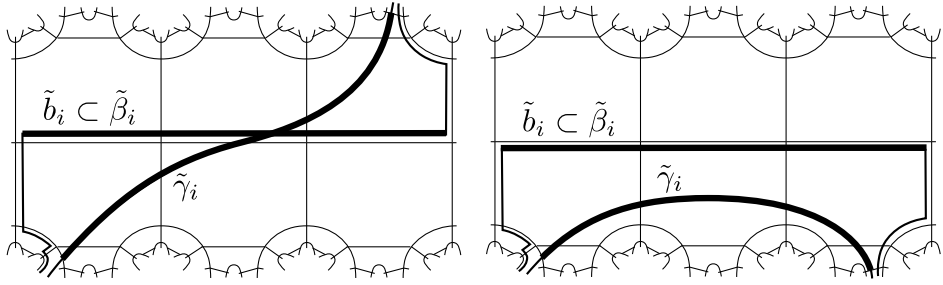}
 \caption[The lift of a boundary subarc]{Lifts of bridging and interior boundary subarcs (in bold) and the lifts of the associated boundary subwords.}
 \label{fig:BoundaryArc}
 \end{figure} 

Because the hexagonal tiling of $\tilde \S$  is invariant under deck transformations, $\gamma_i$ is independent of the lift we choose of $b_i$.

We defined boundary subarcs of $\gamma$ as the projection down to $\S$ of certain subarcs of a lift $\tilde \gamma$ of $\gamma$. So a priori, a boundary subarc $\gamma_i$ need not be a proper subarc of $\gamma$. We now show that boundary subarcs are, in fact, proper subarcs of $\gamma$.
 
 \begin{lem}
 \label{lem:PiHomeo}
  Suppose $w_\Pi(\gamma) = b_1s_1 \dots b_ns_n$ with $n \geq 2$ and take a boundary subarc $\gamma_i = \alpha(b_i)$. Then $l(\gamma_i) \leq l(\gamma)$ for each $i$.
 \end{lem}
 
 \begin{proof}
  Choose a lift $\tilde \gamma$ of $\gamma$. Fix an $i$, and suppose the boundary subword $b_i$ lies on a simple closed geodesic $\beta_i$ in $\S$. Then there is some lift $\tilde \beta_i$ of $\beta_i$ of $\S$ and a lift $\tilde \gamma_i$ of $\gamma_i$ so that
  \[
   \tilde \gamma_i = \tilde \gamma \cap N_k(\tilde \beta_i), \mbox{ where } k = 1 \mbox{ or } 2
  \]
 Note that $\tilde \gamma$ must cross $\partial N_k(\tilde \beta_i)$ in two points, for any $k \in \N$. If not, then $\tilde \gamma$ has an infinite ray that stays a bounded distance away from $\tilde \beta_i$. So either $\gamma$ spirals around $\beta_i$, which means it cannot be closed, or $\gamma = \beta_i$. Both possibilities contradict our assumptions.
 
 \begin{figure}[h!]
  \centering 
  \includegraphics{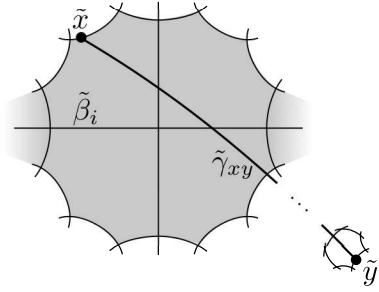}
  \caption{The arc $\tilde \gamma_{xy}$ passes through $N_1(\tilde \beta_i)$, which is shaded gray.}
  \label{fig:K1Subarc}
 \end{figure}

 First, suppose $k = 1$. Then $w_\Pi(\gamma) = b_1 \dots s_n$ for $2 \leq n \leq 5$. See Figure \ref{fig:K1Subarc} for what follows. Let $\tilde x$ be one of the endpoints of $\tilde \gamma_i$ lying on $\partial N_1(\tilde \beta_i)$. Abusing notation, we can view $\gamma$ as a deck transformation that acts by translation along $\tilde \gamma$. Let $\tilde y = \gamma(\tilde x)$. Then let $\tilde \gamma_{xy}$ be the subarc of $\tilde \gamma$ between points $\tilde x$ and $\tilde y$.  We can assume that $\tilde \gamma_i$ and $\tilde \gamma_{xy}$ overlap because if the interior of $\tilde \gamma_{xy}$ does not pass through $N_1(\tilde \gamma)$, then we can let $\tilde y = \gamma^{-1}(\tilde x)$, instead. Since $l(\tilde \gamma_{xy}) = l(\gamma)$, we just need to show that, in fact, $\tilde \gamma_i$ is contained inside $\tilde \gamma_{xy}$.

 Note that $\tilde x$ and $\tilde y$ project to the same point in $\S$. Since $\tilde x$  is an endpoint of $\tilde \gamma_i$, it lies on a hexagon edge on $\partial N_1(\tilde \beta_i)$, which does not touch $\tilde \beta_i$. Thus, $\tilde y$ also lies on an edge that does not touch $\tilde \beta_i$. Therefore $\tilde \gamma_{xy}$ is not contained in $N_1(\tilde \beta_i)$. So $\tilde \gamma_i \subseteq \tilde \gamma_{xy}$.

 Now suppose that $w_\Pi(\gamma) = b_1s_1 \dots b_ns_n$ with $n \geq 6$. Then 
 \[
  \tilde \gamma_j = \tilde \gamma \cap N_2(\tilde \beta_j)
 \]
 for some lift $\tilde \beta_j$ of a curve in $\Pi$, for each $j = 1, \dots, n$. Note that the shortest edge path between $\tilde \beta_i$ and $\tilde \beta_{i+3}$ must have at least 3 seam edges in it. So $\tilde \gamma_i$ and $\tilde \gamma_{i+3}$ overlap inside at most 1 hexagon of the hexagonal tiling of $\tilde \S$. In fact, because $n \geq 6$, no other lift of $\gamma_{i+3}$ overlaps with $\tilde \gamma_i$ in more than 1 hexagon. Thus, $\gamma_i$ and $\gamma_{i+3}$ have at most 2 segments in common (where a segment is a maximal subarc of $\tilde \gamma$ lying in some hexagon of the hexagon decomposition of $\S$. But by definition, $\gamma_j$ has at least 4 segments for each $j$. So at least two segments of $\gamma_{i+3}$ are not segments of $\gamma_i$. Therefore, $l(\gamma_i) \leq l(\gamma)$.
 
\end{proof}

\subsection{Intersection numbers of boundary subarcs approximate $i(\gamma, \gamma)$}

Boundary subarcs do not partition $\gamma$, so it is not true that $\sum i(\gamma_i, \gamma_j) = i(\gamma,\gamma)$. However, this holds up to a multiplicative constant.

\begin{lem}
\label{lem:OverlapBound}
 Let $\gamma \in \G^c$ and let $\Pi$ be a pants decomposition. Let $\gamma_1, \dots, \gamma_n$ be the boundary subarcs of the pair $(\gamma, \Pi)$. Then
 \[
  \sum_{i, j = 1}^n i(\gamma_i, \gamma_j) \leq 25 i(\gamma, \gamma) 
 \]
\end{lem}
\begin{proof}
 Parameterize $\gamma$ by $\gamma : [0,1] \rightarrow \S$. Then the self-intersection points of $\gamma$ are given by times $t_1 < t_2 < \dots < t_n$ and $s_1, \dots, s_n$ so that $\gamma(t_k) = \gamma(s_k)$ for each $k$. For each pair $(t_k, s_k)$, we need to bound the number of pairs $(\gamma_i,\gamma_j)$ so that $t_k$ is in the domain of $\gamma_i$ and $s_k$ is in the domain of $\gamma_j$.
 
 Choose an intersection time $t_k$. Lift $\gamma$ to the universal cover, and consider one lift of $\gamma(t_k)$. This is some point $\tilde x$ on $\tilde \gamma$. We will count the number of boundary subarcs whose lifts contain $\tilde x$.
 
 The point $\tilde x$ lies in some hexagon $h$. So we really just need to count the number of boundary subarcs whose lifts to $\tilde \gamma$ pass through $h$. 
    
\begin{figure}[h!]
 \centering 
 \includegraphics{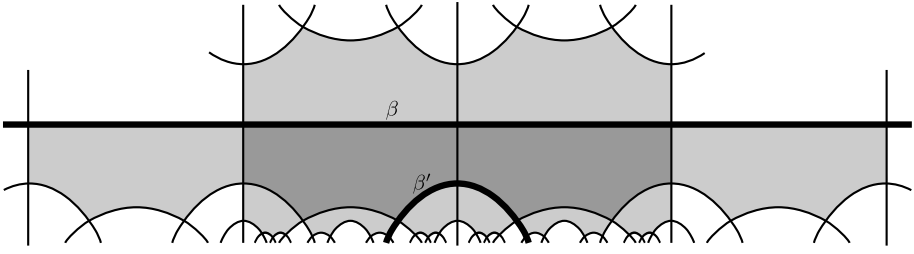}
 \caption{Intersections of the 1- and 2- hexagon neighborhoods of $\beta$ and $\beta'$ are shaded.}
 \label{fig:NbhdIntersection}
\end{figure}

Suppose $\tilde \beta$ and $\tilde \beta'$ are the lifts of two curves in $\Pi$. Consider the intersection $I = N_2(\beta) \cap N_2(\beta')$ of their two hexagon neighborhoods (Figure \ref{fig:NbhdIntersection}). Between the first time $\tilde \gamma$ enters $I$ and last time it leaves, it will pass through at most 4 hexagons. Thus, the lifts of two boundary subarcs overlap in at most 4 hexagons.  

No boundary subarc is contained in any other boundary subarc. Therefore, any hexagon $h$ intersects the lifts of at most 5 boundary subarcs. So any intersection time $t_k$ or $s_k$ lies in the domain of at most 5 boundary subarcs. Therefore, for each pair $(t_k, s_k)$, there are at most 25 pairs $(\gamma_i, \gamma_j)$ so that $t_k$ is in the domain of $\gamma_i$ and $s_k$ is in the domain of $\gamma_j$. 
\end{proof}

\subsection{Two types of boundary subarcs}
Suppose $b_i$ is a boundary subword lying on the simple closed curve $\beta_j \in \Pi$. Recall that we had two types of boundary subwords. If $b_i$ is an interior boundary subword, then we say that $\gamma_i$ is an \textbf{interior} boundary subarc. In this case, $\tilde \gamma_i$ does not intersect $\tilde \beta_j$. 
 
Label the two sides of (a regular neighborhood of) $\beta_j$ by $\beta_j^+$ and $\beta_j^-$. Then the seam edges $s_{i-1}$ and $s_i$ enter and exit the same side of $\beta_j$, respectively. If they enter and exit through $\beta_j^+$, we say that $\gamma_i$ \textbf{lies on the positive side} of $\beta_j$. Otherwise, $\gamma_j$ lies on the negative side. Let
\[
 \Gamma'_{2j} = \{ \gamma_i \ | \ \gamma_i \mbox{ lies on the positive side of } \beta_j\}
 \]
 \[
 \Gamma'_{2j+1} = \{ \gamma_i \ | \ \gamma_i \mbox{ lies on the negative side of } \beta_j\} 
\]

 Now let
 \[
  \Gamma' = \bigcup_{j=1}^{2m} \Gamma'_l
 \]
 This is the set of all interior boundary subarcs of the pair $(\gamma, \Pi)$. 
 
 Otherwise, $b_i$ is a bridging boundary subword, and $\tilde \gamma_i$ intersects $\tilde \beta_j$. Then we call $\gamma_i$ a \textbf{bridging} boundary subarc, or, alternatively, say that $\gamma_i$ \textbf{bridges} $\beta_j$. Let 
 \[
 B_j = \{ \gamma_i \ | \ \gamma_i \mbox{ bridges } \beta_j\}
 \]
 and let 
 \[
  B = \bigcup_{j=1}^m B_j
 \]
 be the set of all bridging boundary subarcs of the pair $(\gamma, \Pi)$.


\section{Intersection bounds for bridging boundary subwords}
\label{sec:BridgingBdrySubwordIntBound}
  We first prove Proposition \ref{prop:BdrySubwordIntBound} for bridging boundary subwords. The formulation is as follows:
  \begin{lem}
  \label{lem:BridgingBdrySubwordIntBound}
   There is a pants decomposition $\Pi$ and some element $f \in \Mod_\S$ so that if $B$ is the set of bridging boundary subarcs of the pair $(f\cdot \gamma, \Pi)$, then
   \[
    \sum_{\gamma_i \in B} |b_i| \lesssim K \mbox{ and } \# B \leq c_\S \sqrt K
   \]
  where $\gamma_i = \alpha(b_i)$ for each $i$, and the constants depend only on the topology of $\S$.
  \end{lem}
  \subsection{Idea of proof} The proof is divided into two major parts, corresponding to the two inequalities in the lemma.

  \textbf{Step 1.} First, we find a pants decomposition so that $w_\Pi(\gamma)$ satisfies 
 \[
  \# B \lesssim \sqrt K
 \]

 Take any pants decomposition $\Pi$. The curve $\gamma$ may cross each curve in $\Pi$ multiple times. The number of times $\gamma$ crosses all the curves in $\Pi$ is exactly the number of bridging boundary subwords (Proposition \ref{prop:BridgingIsCrossing}). So in fact, we find a pants decomposition $\Pi$ so that 
 \[
 i(\gamma, \Pi) \lesssim \sqrt K
 \]
(Proposition \ref{prop:NicePants}).
  
 \begin{figure}[h!]
  \centering 
  \includegraphics{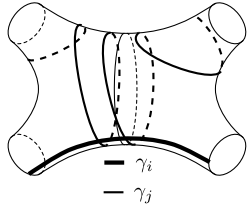}
  \caption{The difference in twisting of $\gamma_i$ and $\gamma_j$ creates intersections between them.}
  \label{fig:OverlapBridging}
 \end{figure}

 \textbf{Step 2.} Next we find a $\gamma' \in \Mod_\S \cdot \gamma$ that differs from $\gamma$ by Dehn twists about $\Pi$ and satisfies:
 \[
  \sum_B |b_i| \lesssim K
 \]
 (Lemma \ref{lem:BridgingTotalLength}.)
 
 As a bridging boundary subarc $\gamma_i$ crosses a curve $\beta_k \in \Pi$, it may spiral around it. Note that a single arc twisting by itself does not contribute anything to intersection number. But suppose two boundary subarcs $\gamma_i$ and $\gamma_j$ bridge some $\beta_k \in \Pi$. If $\gamma_i$ twists $\Delta_i$ times, and $\gamma_j$ twists $\Delta_j$ times, then
 \[
  i(\gamma_i, \gamma_j) \approx |\Delta_i-\Delta_j|
 \]
  where $\Delta_i$ is the signed twisting number of $\gamma_i$ (Figure \ref{fig:OverlapBridging}.) So we get 
 \[
  \sum_{\gamma_i, \gamma_j \i B_k} |\Delta_i - \Delta_j| \lesssim K
 \]
 (Claim \ref{cla:BridgingDifference}.)
 
  We apply Dehn twists about curves in $\Pi$ to $\gamma$ so that the total signed twisting about each $\beta \in \Pi$ is as small as possible (Claim \ref{cla:ApplyDehnTwists}).  This gives us a new curve $\gamma'$. 
 
 Since $\gamma'$ has the smallest possible total twisting about curves in $\Pi$, we can show that, in fact, 
 \[
  \sum_{B_k} |\Delta_i'| \lesssim K
 \]
 
 Lastly, we consider the word $w_\Pi(\gamma') = b_1's_1' \dots b_n' s_n'$. Note that if $|b_i'| = 2m$ then $\gamma_i$ twists roughly $m$ times. Thus, $|\Delta_i'| \approx |b_i'|$ for each bridging boundary subword $b_i$. This gives us that
 \[
  \sum_B |b_i'| \lesssim K
 \]
 where $B$ is the set of bridging boundary subarcs of $\gamma'$ (Claim \ref{cla:BridgingTotalLength}). This gives us Lemma \ref{lem:BridgingBdrySubwordIntBound} (Section \ref{sec:BridgingCombined}).

 \subsection[Choosing nice pants (decompositions)]{Number of subarcs}

 Given any curve $\gamma$ and pants decomposition $\Pi$, let $i(\gamma, \Pi)$ denote the total number of intersections between $\gamma$ and the curves in $\Pi$. So
 \[
  i(\gamma, \Pi) = \sum_{\beta_i \in \Pi} i(\gamma, \beta_i)
 \]

\begin{prop}
\label{prop:BridgingIsCrossing}
Take any curve $\gamma$ and pants decomposition $\Pi$. Let $B$ be the set of bridging boundary subarcs of the pair $(\gamma, \Pi)$. Then
\[
 \# B = i(\gamma, \Pi)
\]

\end{prop}
  
\begin{proof}

 It is easier to see this in the universal cover $\tilde \S$ of $\S$. Take a lift $\tilde \gamma$ of $\gamma$. Suppose $\tilde \gamma$ intersects a lift $\tilde \beta_j$ of $\beta_j \in \Pi$. Let $p(\gamma)$ be the path associated to $w_\Pi(\gamma)$. Lift the homotopy between $p(\gamma)$ and $\gamma$ so that $\gamma$ lifts to $\tilde \gamma$ and $p(\gamma)$ lifts to a curve $\tilde p(\gamma)$. Since $\tilde \gamma$ crosses $\tilde \beta_j$, the curve $\tilde p(\gamma)$ must also cross $\tilde \beta_j$. But $\tilde p(\gamma)$ is made up of arcs that lie on lifts of curves in $\Pi$, and seam arcs connecting curves in $\Pi$. So $\tilde p(\gamma)$ crosses $\tilde \beta_j$ if and only if it has a bridging boundary subword $b_i$, possibly empty, that lies on $\tilde \beta_j$. Therefore, $B$ is in one to one correspondence with intersections between $\gamma$ and $\Pi$.
\end{proof}
Note that if we choose a pants decomposition $\Pi$ at random, then we only have 
\[
i(\gamma, \Pi) \leq l(\gamma) l(\Pi)
\]
where $l(\Pi) = \sum_{\beta \in \Pi} l(\beta)$. This follows, for example, by an argument similar to the proof of \cite[Theorem 1.1]{Basmajian13}. Because of this, the following proposition may be of independent interest.

\begingroup
\def \thetheorem{\ref{prop:NicePants}}
  \begin{prop}
  Let $\gamma \in \G^c$ with $i(\gamma, \gamma) = K \geq 1$. Then there exists a pants decomposition $\Pi$ so that
  \[
   i(\gamma, \Pi) \leq c_\S \sqrt K
  \]
 where the constant $c_\S$ depends only on the topology of $\S$.
 \end{prop}
 \endgroup
 
 The author would like to thank Kasra Rafi for the conversation in which we came up with the main idea of this proof.
\begin{proof}
 We find the pants decomposition $\Pi$ one curve at a time. Each time we add a simple closed curve $\beta$ to $\Pi$, we cut along $\beta$. So $\S$ gets progressively cut along simple closed curves until we decompose $\S$ into the union of pairs of pants. 
 
 So suppose we have a (connected) surface $\S'$ with $b$ boundary components, that either contains a single non-simple closed geodesic or that is traversed by at most $\sqrt K$ geodesic arcs connecting its boundary components, so that these arcs intersect each other at most $K$ times. Let $\gamma$ denote this single closed geodesic or collection of geodesic arcs. Then we find a simple closed curve $\beta$ that crosses $\gamma$ at most $ 57 \sqrt K$ times. We do this as follows:
 
 \begin{figure}[h!]
  \centering
  \includegraphics{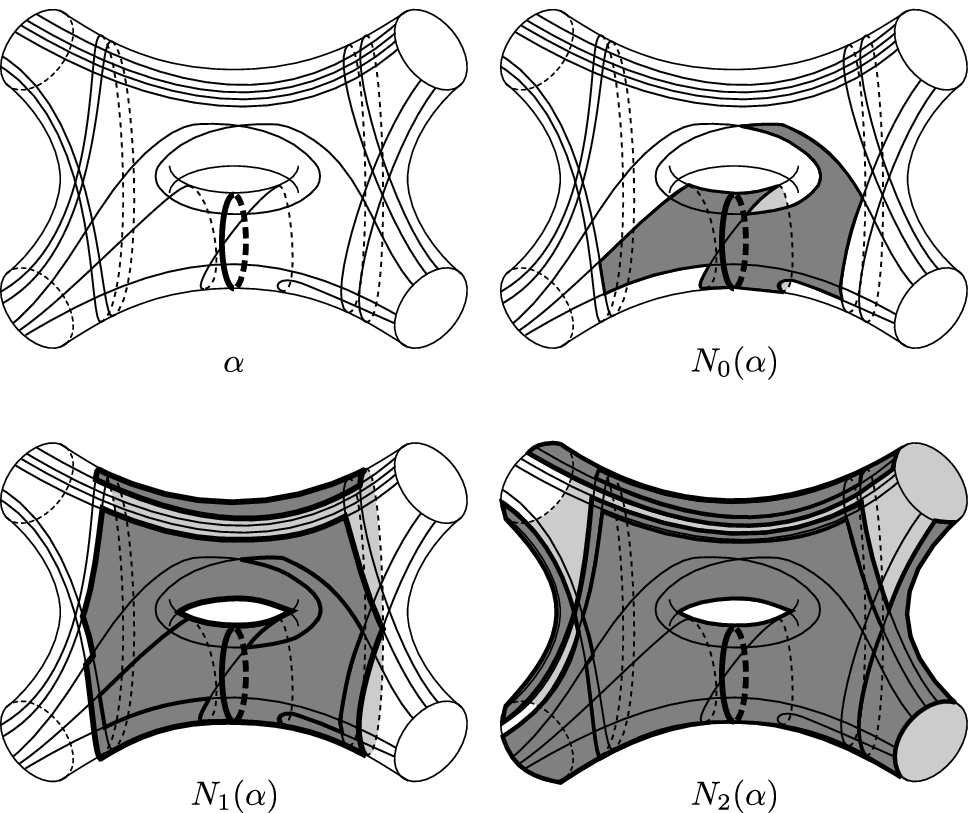}
  \caption{}
  \label{fig:N_i}
 \end{figure}
 
 \textbf{Step 1.} We choose an essential, non-peripheral, non-separating simple closed curve $\alpha$ so that $i(\alpha, \gamma)$ is as small as possible.  For example, in Figure \ref{fig:N_i}, $\alpha$ intersects $\gamma$ 3 times. If $\S$ has genus 0, then we drop the non-separating condition, as in Figure \ref{fig:Separating}.
  
 \textbf{Step 2.} The curve $\gamma$ cuts $\S$ into regions. An Euler characteristic argument implies that $\sqrt K$ arcs with $K$ total self-intersections cut $\S$ into at most $K + \sqrt K$ regions. The number of regions does not need to be very precise. As $\sqrt K \leq K$, we can say we have at most $2 K$ regions.
 
 \begin{rem}
 \label{rem:RegionIntNumber}
  Suppose an arc $\eta$ passes through $i+1$ regions. Then it crosses from one region to another $i$ times. Since $\eta$ can pass through region corners, and at most 4 regions can meet at each corner, this means that
  \[
   \# \eta \cap \gamma \leq 2i
  \]
 where we do not count any intersections between the endpoints of $\eta$ with $\gamma$.
 \end{rem}

 This motivates the following definition. We define neighborhoods $N_0(\alpha) \subset N_1(\alpha) \subset \dots \subset N_i(\alpha) \subset \dots $, where $N_0(\alpha)$ contains all the regions that touch $\alpha$, and $N_{i+1}(\alpha)$ contains all the regions that touch $N_i(\alpha)$ (in either an edge or a corner.) (Figure \ref{fig:N_i}).
 
 A small caveat is that we want $N_i(\alpha)$ to have essential boundary for each $i$. So suppose $N'_{i+1}(\alpha)$ is the union of $N_i(\alpha)$ and the set of regions that touch $N_i(\alpha)$. Then let $N_{i+1}(\alpha)$ be $N'_{i+1}(\alpha)$ together with all contractible subsets of $\S \setminus N'_{i+1}(\alpha)$. 
 
 \begin{rem} 
 \label{rem:ConnectingRegionToAlpha}
 Any point $x$ on the boundary of $N_i(\alpha)$ can be joined to $\alpha$ by an arc $\eta: x \mapsto \alpha$ that passes through at most $i+1$ regions. In particular, it will intersect $\gamma$ at most $2i$ times.
 \end{rem}

 \textbf{Step 3.} There are finitely many regions, so at some point, there will be an $n$ for which one of the following two things will happen:
 \begin{enumerate}
  \item The curve $\alpha$ separates $N_n(\alpha)$ but not $N_{n+1}(\alpha)$. For example, in Figure \ref{fig:N_i}, $\alpha$ separates $N_0(\alpha)$ but not $N_1(\alpha)$.
  
  Since the nested sequence of neighborhoods $N_0(\alpha) \subset \dots \subset N_i(\alpha) \subset \dots$ eventually fills out the entire surface, either this condition is eventually satisfied, or $\alpha$ separates $\S$ (which is only possible if $\S$ has genus 0.)
  
  \item The number of new regions in $N_{n+1}(\alpha)$ will be fewer than $4 \sqrt K$:
  \[
   \# N_{n+1}(\alpha) \setminus N_n(\alpha) \leq 4 \sqrt K
  \]
  For example, in Figure \ref{fig:N_i}, $N_0(\alpha)$ only has 3 regions. Since $i(\gamma, \gamma) = 39$, this condition is satisfied for $n = -1$ (using the convention that $N_{-1}(\alpha) = \alpha$). 
 \end{enumerate}
 
 We let $n$ be the least number so that one of these two conditions are satisfied. In Figure \ref{fig:N_i}, $n = -1$, because condition 2 is satisfied before condition 1. 
 
 In particular, since condition 2 must be satisfied at some point, and since there are at most $2K$ regions, we must have that
 \[
  n \leq \frac 12 \sqrt K
 \]
 
 We will examine what happens when each of the conditions fails first.
 
 \textbf{Case 1.} Suppose Condition 1 fails first. That is, $\alpha$ separates $N_n(\alpha)$, but does not separate $N_{n+1}(\alpha)$. 
 
 In this case, we construct another essential, non-peripheral, non-separating curve $\beta$ so that 
 \[
  i(\beta, \gamma) \leq 2 \sqrt K + 2 + \frac 12 i(\alpha, \gamma)
 \]
 We chose $\alpha$ so that $i(\alpha, \gamma) \leq i(\beta,\gamma)$. Thus the above inequality implies that 
 \[
 i(\alpha, \gamma) \leq 4 \sqrt K + 4
 \]
 
 The curve $\beta$ will be the concatenation of an arc $\rho$ in $\S \setminus \alpha$ with a subarc $\alpha'$ of $\alpha$. First we build the arc $\rho$ that will join $\alpha$ to itself. So cut $N_{n+1}(\alpha)$ along $\alpha$. The result is connected by assumption, and has two boundary components, $\alpha_+$ and $\alpha_-$, that come from $\alpha$.

 \begin{figure}[h!]
  \centering
  \includegraphics{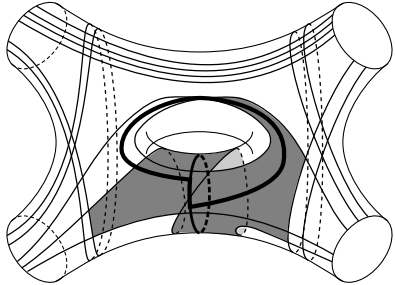}
  \caption{Since $\alpha$ does not separate $N_2(\alpha)$, we can find this curve $\beta$.}
  \label{fig:NonSeparating}
 \end{figure}
 Then we claim that there is a curve $\rho$ joining $\alpha_+$ to $\alpha_-$ that passes through at most $\sqrt K + 2$ regions. There are at most 2 regions, $R_1, R_2$, in $N_{n+1}(\alpha) \setminus N_n(\alpha)$, so that $R_1 \cup R_2$ joins one component of $N_n(\alpha) \setminus \alpha$ to the other. In Figure \ref{fig:NonSeparating}, there is actually a single region in $N_1(\alpha) \setminus N_0(\alpha)$ that connects the two components of $N_0(\alpha) \setminus \alpha$. Take any point $x$ where $R_1$ meets the boundary of $N_n(\alpha)$. Then we can join $x$ to $\alpha$ by an arc that passes through at most $n$ regions. Thus, we can join $\alpha_+$ to $\alpha_-$ by an arc $\rho$ that passes through at most $2n+2$ regions. As remarked above, $n \leq \frac 12 \sqrt K$, so $\rho$ passes through at most $ \sqrt K + 2$ regions. By Remark \ref{rem:RegionIntNumber}, this means that
 \[
  \# \rho \cap \gamma \leq 2 \sqrt K + 2
 \]

 Now we think of $\rho$ as an arc in $\S$. Its endpoints lie on $\alpha$. Thus, we can join the endpoints of $\rho$ by a subarc $\alpha'$ of $\alpha$ so that 
 \[
 \# \alpha' \cap \gamma \leq \frac 12 i(\alpha, \gamma)
 \]
  Take the concatenation $ \beta = \rho \circ \alpha$. Then we have
 \[
  i(\beta, \gamma) \leq 2 \sqrt K + 2 + \frac 12 i(\alpha, \gamma)
 \]
 Note that by construction, $i(\beta, \alpha) = 1$. This means that $\beta$ is essential, non-peripheral, and non-separating. So, as explained above, we get that $i(\alpha, \gamma) \leq 4 \sqrt K + 4$.
 
 \textbf{Case 2.}  Now suppose Condition 2 fails first. That is, $\alpha$ separates $N_{n+1}(\alpha)$, but the number of regions in $N_{n+1}(\alpha) \setminus N_n(\alpha)$ is fewer than $4 \sqrt K$. For example, in Figure \ref{fig:Separating}, $\alpha$ actually separates $\S$, which is a 4-holed sphere. The multi-arc $\gamma$ has 16 self-intersections, and $N_2(\alpha)$ has 4 more regions than $N_1(\alpha)$.
 
 \begin{figure}[h!]
  \centering 
  \includegraphics{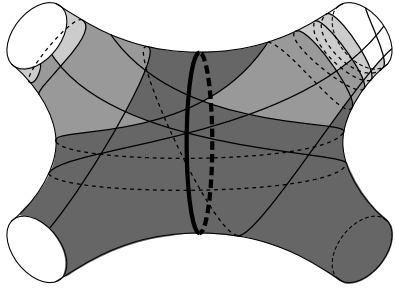}
  \caption{$N_0(\alpha), N_1(\alpha)$ and $N_2(\alpha)$ are shown in different shades of gray.}
  \label{fig:Separating}
 \end{figure}
 
 In this case, let $\beta_1, \dots ,\beta_m$ be the boundary components of $N_{n+1}(\alpha)$. Consider one such boundary component $\beta_i$. Then any point on $\beta_i$ either lies on the boundary of a region that is new to $N_{n+1}(\alpha)$, or it lies on the boundary of $\S$. There are at most $4 \sqrt K$ regions in $N_{n+1}(\alpha) \setminus N_n(\alpha)$, and $2\sqrt K$ regions touching $\partial \S$ (because $\gamma$ consists of at most $\sqrt K$ arcs, which cut $\partial \S$ into $2 \sqrt K$ pieces). Thus, $\beta_i$ passes through the boundary of at most $6 \sqrt K$ regions. Because regions can meet in corners, this implies that
 \[
  i(\beta_i, \gamma) \leq 12 \sqrt K
 \]
 for each boundary component $\beta_i$ in $N_{n+1}(\alpha)$. 
 
 \begin{figure}[h!]
  \centering 
  \includegraphics{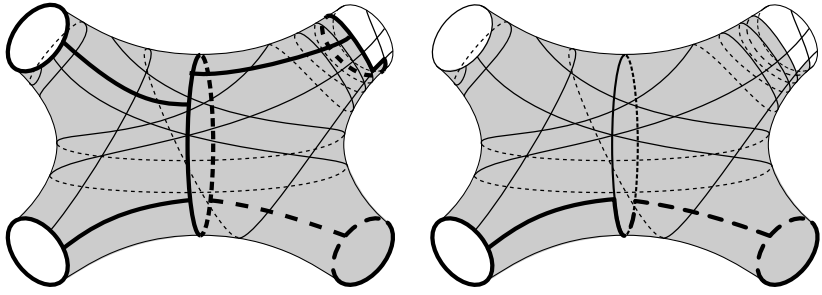}
  \caption{A path $\rho_i$ connects its boundary component $\beta_i$ to $\alpha$, for each $i = 1,2,3,4$. Then, we form the curve $\beta$.}
  \label{fig:PathsToBoundary}
 \end{figure}

 If $N_{n+1}(\alpha)$ has a non-peripheral boundary component $\beta_i$, then we are done: we have found an essential, simple closed curve $\beta_i$ that intersects $\gamma$ at most $12 \sqrt K$ times. So suppose that all boundary components of $N_{n+1}(\alpha)$ are peripheral, as in Figure \ref{fig:PathsToBoundary}.  In this case, $N_{n+1}(\alpha)$ is homeomorphic to $\S$. As we assume that $\alpha$ separates $N_{n+1}(\alpha)$, it must separate $\S$. We only allow this if $\S$ has genus 0.
 
 In this case, we will find an essential, non-peripheral closed curve $\beta$ so that
 \[
  i(\beta, \gamma) \leq 28 \sqrt K + 8 + \frac 12 i(\alpha, \gamma)
 \]
 When $\S$ is genus 0, we assume that $\alpha$ is the essential, non-peripheral closed curve that intersects $\gamma$ the least. So $i(\alpha, \gamma) \leq i(\beta, \gamma)$ implies that
 \[
  i(\alpha, \gamma) \leq 56 \sqrt K + 16
 \]
 
 We build $\beta$ as follows. $\S$ has genus 0 with $b$ boundary components, for $b \geq 4$. So we label the boundary components of $N_{n+1}(\alpha)$ by $\beta_1, \dots, \beta_b$. By Remark \ref{rem:ConnectingRegionToAlpha}, any point $x_i$ on $\beta_i$ can be joined to $\alpha$ by an arc $\rho_i$ with $\# \rho_i \cap \gamma \leq 2n+2$, as in the left-hand side of Figure \ref{fig:PathsToBoundary}. Since $n \leq \frac 12 \sqrt K$, we have
 \[
  \# \rho_i \cap \gamma \leq \sqrt K +2
 \]
Let $y_i$ be the endpoint of $\rho_i$ that lies on $\alpha$. Since $b \geq 4$, there are at least 4 such points. So without loss of generality, $y_1$ and $y_2$ can be joined by a subarc $\alpha' \subset \alpha$ so that
\[
  \# \alpha' \cap \gamma \leq \frac 14 i(\alpha, \gamma)
\]

 The arc $\rho = \rho_1 \circ \alpha' \circ \rho_2$ joins $\beta_1$ to $\beta_2$, as in the right-hand side of Figure \ref{fig:PathsToBoundary}. By the above,
 \[
  \# \rho \cap \gamma \leq 2 \sqrt K + 4 + \frac 14 i(\alpha, \gamma)
 \]
Note that it must be a simple arc, because if, for example, $\rho_1$ and $\rho_2$ intersect, we can do surgery on one of them so that it goes through strictly fewer regions. 

We can then form the simple closed curve 
 \[
  \beta = \rho \circ \beta_1 \circ \rho^{-1} \circ \beta_2
 \]
 (See the right-hand side of Figure \ref{fig:PathsToBoundary}.) We know that $i(\beta_i, \gamma) \leq 12 \sqrt K$ for each $i = 1, \dots, b$. So,
 \[
  i(\beta, \gamma) \leq 28 \sqrt K + 8 + \frac 12 i(\alpha, \gamma)
 \]
 Because $\beta_1$ and $\beta_2$ are homotopic to distinct boundary components of $\S$, the curve $\beta$ is essential and non-peripheral. So by the argument above, $i(\alpha, \gamma) \leq 56 \sqrt K + 16$.
 
 \textbf{Step 4.} In each case above, we found an essential, simple closed curve $\beta$ that crosses $\gamma$ at most $56 \sqrt K + 16$ times. As $K \geq 1$, this means $i(\gamma, \beta) \leq 57 \sqrt K$.

 \textbf{Step 5:}
 Now take our original surface $\S$ and a single non-simple closed geodesic $\gamma$ with at most $K$ self-intersections. Then there is a simple, essential, non-peripheral closed curve $\beta_1$ on $\S$ with
 \[
   i(\gamma, \beta_1) \leq 57 \sqrt K
 \]
 
 Now suppose we have distinct simple closed curves $\beta_1, \dots, \beta_i$ with $i(\beta_i, \gamma) = I_i$. Cutting $\gamma$ along $\beta_1, \dots, \beta_i$ gives us a multiarc $\gamma_i$ composed of 
 \[
  T_i = I_1 + \dots + I_i
 \]
 arcs. As $I_1 = 57 \sqrt K$, we see that $T_i \geq \sqrt K$. So we can use the argument in Steps 1-4 to find a curve $\beta_{i+1}$ with
 \[
  i(\gamma, \beta_{i+1}) \leq 57 T_i
 \]
 since we can use $(T_i)^2$ as the new upper bound for $i(\gamma_i, \gamma_i)$. By induction, we see that 
 \[
  T_i = 57(1+57)^{i-1} \sqrt K
 \]
 We continue finding curves until we get a pants decomposition $\Pi = \{\beta_1, \dots, \beta_m\}$. Note that $i(\gamma,\Pi) \leq T_m$. Thus, we can simplify the above bound to get 
 \[
  i(\gamma, \Pi) \leq 58^m \sqrt K
 \]

 \end{proof}
 
 \subsection{Total length of bridging boundary subwords}
 \label{sec:BridgingTotalLength}
 From now on, we will assume that $\Pi = \{\beta_1, \dots, \beta_m\}$ is a pants decomposition so that $i(\gamma, \Pi) \leq c_\S \sqrt K$ for all $i$. Let $\tau_i$ be the Dehn twist about $\beta_i$.
 
 We show the following:
 \begin{lem}
 \label{lem:BridgingTotalLength}
  There is some product $f = \tau_1^{n_1} \dots \tau_m^{n_m}$ of Dehn twists so that if $B$ is the set of bridging boundary subarcs of the pair $(f \cdot \gamma, \Pi)$ then
  \[
   \sum_{\gamma_i \in B} |b_i| \lesssim K
  \]
 where the constant depends only on $\S$.
 \end{lem}
 Note that if $f$ is any product of Dehn twists about curves in $\Pi$, then $i(f \cdot \gamma, \Pi) \leq c_\S \sqrt K$ still holds. That is, $\Pi$ is the ``right'' pants decomposition for $\gamma$ if and only if it is the ``right'' pants decomposition for $f \cdot \gamma$. So Proposition \ref{prop:NicePants} and Lemma \ref{lem:BridgingTotalLength} together imply Lemma \ref{lem:BridgingBdrySubwordIntBound}.

 \begin{proof}

 Lift the hexagon decomposition of $\S$ to a hexagon decomposition of its universal cover $\tilde \S$. Let $\beta_l \in \Pi$. Choose a lift $\tilde \beta_l$ of $\beta_l$ to $\tilde \S$. Number the hexagons on either side of $\tilde \beta_l$ (Figure \ref{fig:HexagonOrder}). Let $\{h_i\}_{i \in \Z}$ be the hexagons on one side and $\{\chi_i\}_{i \in \Z}$ be the hexagons on the other side, so that $h_i$ and $\chi_i$ share an edge, and $h_i$ is adjacent to $h_{i-1}$ and $h_{i+1}$.
\begin{figure}[h!]
 \centering
 \includegraphics{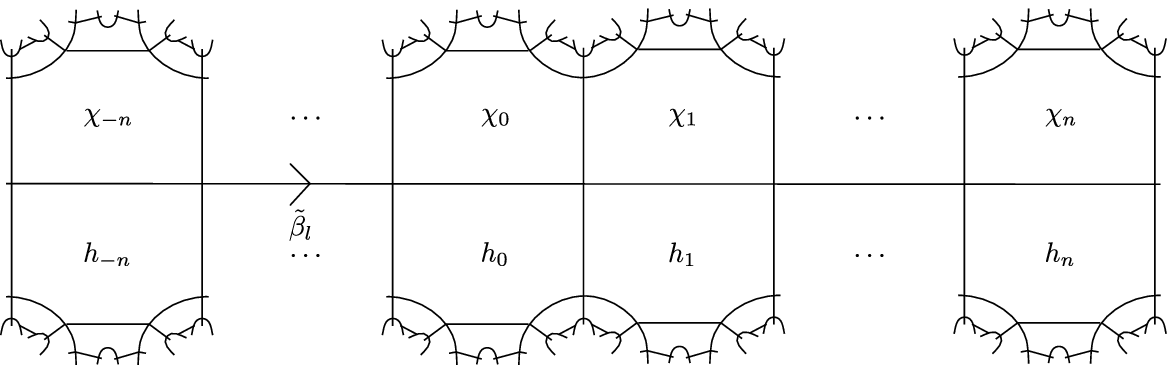}
 \caption{ }
 \label{fig:HexagonOrder}
\end{figure}

Suppose $\tilde \gamma$ is a lift of $\gamma$ that intersects $\tilde \beta_l$. Then there is a boundary subarc $\gamma_i$ with a lift
\[
 \tilde \gamma_i = \tilde \gamma \cap N_k(\tilde \beta_l)
\]
for $k = 1$ or 2. Suppose $\tilde \gamma_i$ enters $N_1(\tilde \beta_l)$ at $h_{n_i}$ and exits at $\chi_{m_i}$.  Let 
\[
 \Delta_i = m_i - n_i 
\]
We will call $\Delta_i$ the \textbf{twisting parameter} of $\gamma_i$, since it is related to the number of times $\gamma_i$ twists about $\beta_l$, and the direction it twists in.

Suppose $\gamma_j$ is another boundary subarc that bridges $\beta_l$. Then there is another lift $\tilde \gamma'$ of $\gamma$ so that
\[
 \tilde \gamma_j = \tilde \gamma' \cap N_k(\beta_l)
\]
is a lift of $\gamma_j$. Then, up to reversing orientation, $\tilde \gamma_j$ enters $N_1(\tilde \beta_j)$ at $h_{n_j}$ and exits at $\chi_{m_j}$ and has twisting parameter $\Delta_j$ (Figure \ref{fig:BridgingIntersections}.)
\begin{cla}
\label{cla:BridgingDifference}
Suppose $i(\gamma, \Pi) \leq c_\S \sqrt K$. If $B_l$ is the set of boundary subarcs of the pair $(\gamma, \Pi)$ that bridge $\beta_l \in \Pi$ for each $l$, then
 \[
 \sum_{l=1}^m \sum_{\gamma_i, \gamma_j \in B_l} |\Delta_i - \Delta_j| \lesssim K
\]
where the constant depends only on the topology of $\S$, and $\Delta_i$ is the twisting parameter of $\gamma_i$.
\end{cla}
\begin{proof}

Both $\tilde \gamma_i$ and $\tilde \gamma_j$ cut $N_1(\tilde \beta_l)$ into two pieces. For example, consider the two components of $N_1(\tilde \beta_l) \setminus \gamma_i$. Whenever $n > n_i$ and $m < m_i$, the edges of $\partial N_1(\tilde \beta_l)$ adjacent to hexagon $h_n$ are in a different component than the edges adjacent to hexagon $\chi_m$. In fact, if 
\begin{align*}
 n_i < n_j & \mbox{ and } m_j < m_i \mbox{ or }  \\
 n_j < n_i & \mbox{ and } m_i < m_j
\end{align*}
we have that
\[
 \# \tilde \gamma_i \cap \tilde \gamma_j = 1
\]
(See Figure \ref{fig:BridgingIntersections}.)
\begin{figure}[h!]
 \centering 
 \includegraphics{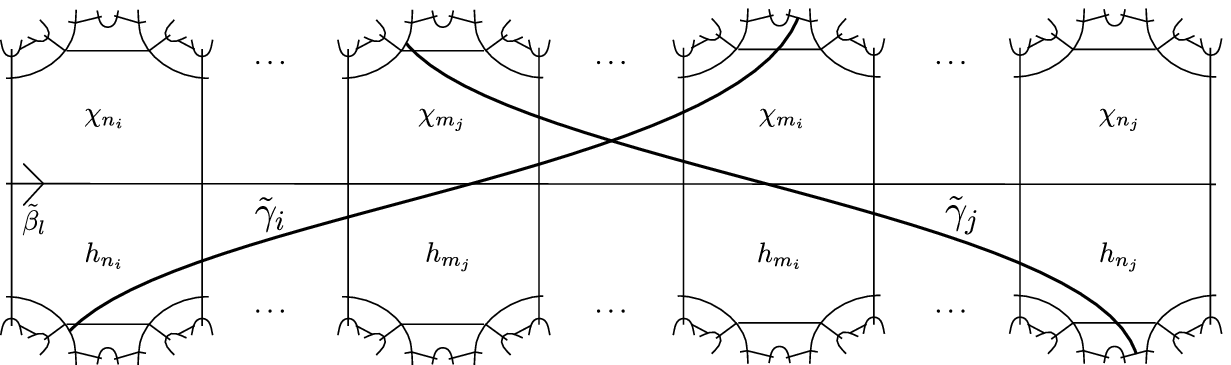}
 \caption{}
 \label{fig:BridgingIntersections}
\end{figure}

Let $f_l$ be the deck transformation that acts by translation along $\tilde \beta_l$ with translation length $l(\beta_l)$. We will use $f_l$ to slide $\tilde \gamma_i$ around and create intersections with $\tilde \gamma_j$.

Up to replacing $f_l$ with $f_l^{-1}$, we have $f_l(h_n) = h_{n+2}$ and $f_l(\chi_n) = \chi_{n+2}$. So $f^k_l(\tilde \gamma_i)$ has endpoints in hexagons $h_{n_i + 2k}$ and $\chi_{m_i + 2k}$. Thus, whenever
\begin{align*}
 n_{i+2k} < n_j & \mbox{ and } m_j < m_{i+2k} \mbox{ or }  \\
 n_j < n_{i+2k} & \mbox{ and } m_{i+2k} < m_j
\end{align*}
we have that
\[
 \# f^k_l (\tilde \gamma_i) \cap \tilde \gamma_j = 1
\]

Recall that we defined the twisting parameter
\[
 \Delta_i = m_i -n_i
\]
for each $i$. Then we can count the number of powers $k$ that result in an intersection:
\[
 \#\{ k \in \Z \ | \ \# f^k_l (\tilde \gamma_i) \cap \tilde \gamma_j = 1 \} \geq \frac 12 |\Delta_i - \Delta_j| - 1
\]

Project all intersections between shifts $f^k_l(\tilde \gamma_i)$ and $\tilde \gamma_j$ down to $\gamma$. Note that they must all project down to distinct self-intersection points of $\gamma$. Thus, 
\[
 i(\gamma_i,\gamma_j) \geq \frac 12 |\Delta_i - \Delta_j| - 1
\]
for any $\gamma_i$ and $\gamma_j$ that bridge $\beta_l$.
By Lemma \ref{lem:OverlapBound}, this implies that
\[
 \sum_{l = 1}^m \sum_{\gamma_i, \gamma_j \in B_l} (|\Delta_i - \Delta_j| - 2) \leq 50 i(\gamma, \gamma)
\]
where $B_l$ is the set of all boundary subarcs of the pair $(\gamma, \Pi)$ that bridge $\beta_l$. We chose $\Pi$ so that $\sum \# B_l \leq c_\S \sqrt K$ for each $l$. Since $i(\gamma,\gamma) \leq K$, this implies
\[
 \sum_{l=1}^m \sum_{\gamma_i, \gamma_j \in B_l} |\Delta_i - \Delta_j| \leq (50 + 2(c_\S)^2) K
\]
\end{proof}

Whenever $\gamma$ satisfies
\begin{equation}
 \label{eq:DeltaBalanced}
 0 \leq \sum \Delta_i  \leq 4 c_\S \sqrt K
\end{equation}
we get the following claim.

\begin{cla}
\label{cla:BridgingTotalLength}
 Suppose $i(\gamma, \Pi) \leq c_\S \sqrt K$. If $\gamma$ also satisfies (\ref{eq:DeltaBalanced}), then 
 \[
  \sum_{\gamma_i \in B} |b_i| \lesssim K
 \]
 where $B$ is the set of bridging boundary subarcs of the pair $(\gamma, \Pi)$, $\gamma_i = \alpha(b_i)$ for each $i$, and the constant depends only on $\S$.
\end{cla}
After we show this claim, we will show how to find a composition of Dehn twists $f \in \Mod_\S$ so that (\ref{eq:DeltaBalanced}) holds for $f \cdot \gamma$.

\begin{proof}
We will first bound $\sum |\Delta_i - \Delta_j|$ from Claim \ref{cla:BridgingDifference} from below by $\sum |\Delta_i|$. Then we will show that $|\Delta_i| \geq |b_i| +1$. Combined with Claim \ref{cla:BridgingDifference}, this will complete the proof.

Let $B_l$ be the set of boundary subarcs of the pair $(\gamma, \Pi)$ that bridge $\beta_l$. Renumber the elements of $B_l$ by $\gamma_1, \dots, \gamma_N$ so that their respective twisting parameters satisfy $\Delta_1 \leq \Delta_2 \leq \dots \leq \Delta_N$. Consider the set of vectors $\{v_i = (1,\Delta_i)\} \subset \R^2$.  Form a convex polygon $P$ with vertices at $w_0 = \mathbf 0$, and $w_i = v_1 + \dots + v_i$ for $i = 1, \dots, N$.
(Figure \ref{fig:Polygon}.)
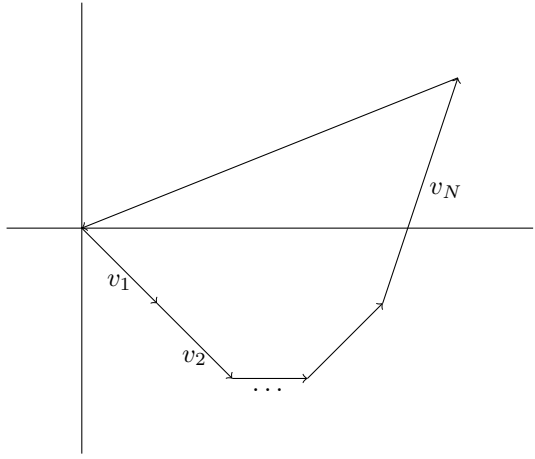
\begin{figure}[h!]
\centering 
 \begin{tikzpicture}
  \draw (-1,0) -- (6,0) (0,-3) -- (0,3);
  \draw[->] (0,0) -- (1,-1);
  \draw[->] (1,-1)-- (2,-2);
  \draw[->] (2,-2) -- (3,-2);
  \draw[->] (3,-2) --(4,-1) ;
  \draw[->] (4,-1) -- (5,2);
  \draw[->] (5,2) -- (0,0);
  
  \path(.5,-.5) node[left ,below]{$v_1$};
  \path(1.5,-1.5) node[left, below]{$v_2$};
  \path(2.5, -2) node[below]{$\dots$};
  \path(4.5, .5)node[right]{$v_N$};
 \end{tikzpicture}
\caption{The polygon $P$. The slopes of its sides are the twisting parameters of $\gamma$ about $\beta_l$.}
\label{fig:Polygon}
\end{figure}

   The area of $P$ is exactly $\sum_{i< j} |\Delta_i - \Delta_j|$. To see this, consider the triangle with vertices at $w_i, w_{i+1}$ and $w_N$.

Two sides of this triangle are given by the vectors $v_i = (1,\Delta_i)$ and $v_{i+1} + \dots + v_N = (N-i, \Delta_{i+1} + \dots + \Delta_N)$, respectively. Thus, its area is given by absolute value of the determinant
   \[
    \begin{vmatrix}
     1 & N - i \\
     \Delta_i & \Delta_{i+1} + \dots + \Delta_N
    \end{vmatrix}
   = \sum_{j = i}^N (\Delta_j - \Delta_i)
   \]
   These triangles are disjoint for all $i$, so the sum of their areas gives the area of $P$. Thus, 
   \[
    area(P) = \sum_{i \leq j} \Delta_j - \Delta_i
   \]
   where we do not need absolute value signs because $\Delta_i \leq \Delta_j$ for $j \geq i$.
   
   We would like to thank Ser-Wei Fu for introducing us to the above technique. Specifically, he showed us how to write a sum of differences of twisting numbers as the area of a polygon, like we do above.
   
    We will bound $area(P)$ from below by a multiple of $\sum |\Delta_i|$.
  Let $T$ be the triangle with vertices $0, a$ and $b$ where $a$ and $b$ are defined as follows:  Let $a = w_N$ be the right-most vertex of $P$, and let $b$ be a point on $P$ with the least $y$-coordinate. As $T$ is contained inside $P$, we will, in fact, bound $area(T)$ from below in terms of $\sum |\Delta_i|$.

  \begin{figure}[h!]
\centering 
 \begin{tikzpicture}
  \draw (-1,0) -- (6,0) (0,-3) -- (0,3);
  \draw[->] (0,0) -- (1,-1);
  \draw[->] (1,-1)-- (2,-2);
  \draw[->] (2,-2) -- (3,-2);
  \draw[->] (3,-2) --(4,-1) ;
  \draw[->] (4,-1) -- (5,2);
  \draw[->] (5,2) -- (0,0);
  
 \draw[draw=black, fill=gray, fill opacity=0.2] (0,0) -- (2.7,-2) -- (5,2) -- (0,0);

 \path (-.1,0) node[left,above]{$0$};
 \path (2.7,-2) node[below]{$b$};
 \path (5,2) node[right,above]{$a$};
 \filldraw (2.7,-2) circle(.05);
 \filldraw (0,0) circle(.05);
 \filldraw (5,2) circle(.05);
 
 \end{tikzpicture}
\caption{The triangle $T$ inside $P$.}
\label{fig:PolygonHeights}
\end{figure}
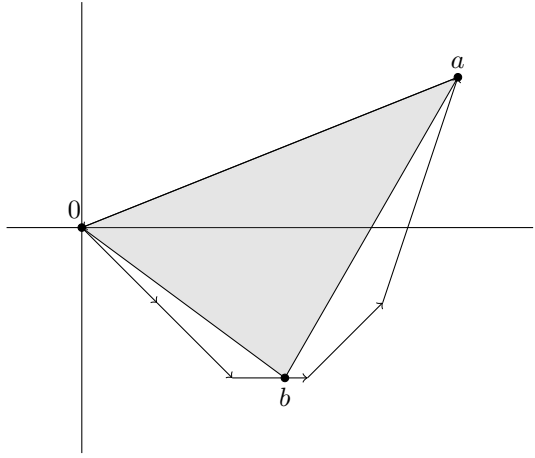
 
We will use Heron's formula to estimate the area of $T$. Heron's formula says
\[
 (area(T))^2 = s (s-l_1)(s-l_2)(s-l_3)
\]
where $l_1, l_2,$ and $l_3$ are the side lengths of $T$, and $s = \frac 12(l_1+l_2+l_3)$. Let 
\[
 l_1 = |0a|, l_2 = |0b| \text{ and } l_3 = |ab|
\]

By the triangle inequality, all four terms in Heron's formula are positive. So we will use that $(area(T))^2 \geq s (s-l_1)$. Thus, we can bound $area(T)$ from below if we bound $s$ from below and $l_1$ from above.

So we need the following bounds on $l_1, l_2$ and $l_3$. Because $a = (N, \sum \Delta_i)$, we have that $N \leq l_1 \leq N + \sum \Delta_i$. Since $1 \leq N \leq c_\S \sqrt K$, and $\sum \Delta_i \leq 4 c_\S \sqrt K$, we have
\[
 1 \leq l_1 \leq 5c_\S \sqrt K
\]

Next, note that both $\overline{0b}$ and $\overline{ab}$ join $b$ to a point with non-negative $y$-coordinate. So if $b = (b_1, b_2)$, then
\[
 l_2 \geq |b_2| \text{ and } l_3 \geq |b_2|
\]
So we will bound $l_2$ and $l_3$ from below if we can bound $|b_2|$ from below. Since $b$ is the lowest point on $P$, its $y$-coordinate must be
 \[
   b_2 = \sum_{\Delta_i < 0} \Delta_i
 \]
 As $0 \leq \sum \Delta_i  \leq 4 c_\S \sqrt K$, we have that 
 \[
   0 \leq  \sum_{\Delta_i > 0} |\Delta_i| - \sum_{\Delta_i < 0} |\Delta_i| \leq 4 c_\S \sqrt K
 \]
 By adding $2 \sum_{\Delta_i < 0} |\Delta_i|$ to both sides and rearranging the resulting inequality, we get
 \[
   |b_2| \geq \frac 12 \sum_{i=1}^N |\Delta_i| - 2 c_\S \sqrt K
 \]
Therefore, we get the inequalities
\[
 l_2 \geq \frac 12 \sum_{i=1}^N |\Delta_i| - 2 c_\S \sqrt K
\]
and
\[
 l_3 \geq \frac 12 \sum_{i=1}^N |\Delta_i| - 2 c_\S \sqrt K
\]

Now we can estimate $s$ and $s - l_1$ as follows:
\[
 s \geq \frac 12 \sum_{i=1}^N |\Delta_i| - 2 c_\S \sqrt K
\]
and
\[
 s -  l_1 \geq \frac 12 \sum_{i=1}^N |\Delta_i| - 7 c_\S \sqrt K
\]
where we ignore the contribution of $+1$ to $s$ from $l_1$ to make our computations cleaner. 

Clearing out the fraction and applying Heron's formula, this gives us that
\begin{align*}
 4 \cdot area(T) & \geq (\sum_{i=1}^N |\Delta_i| - 4 c_\S \sqrt K)(\sum_{i=1}^N |\Delta_i| - 14 c_\S \sqrt K)\\
  & \geq (\sum_{i=1}^N |\Delta_i| - 14 c_\S \sqrt K)^2
\end{align*}
Note that without loss of generality, the terms in the product are both positive. If either term is negative, then $\sum_{i=1}^N |\Delta_i| \leq 14 c_\S \sqrt K$. But this implies that $\sum_{i=1}^N |\Delta_i| \lesssim K$, which is what we are trying to show.

So we can bound $\sum |\Delta_i|$ as follows:
\begin{align*}
 \sum_{i=1}^N |\Delta_i| & \leq 2 \cdot area(T) + 14 c_\S \sqrt K \\
  & \leq 2 \cdot area(P) + 14 c_\S \sqrt K \\
  & = 2 \sum_{i \leq j} \Delta_j - \Delta_i + 14 c_\S \sqrt K \\
  & \lesssim K
\end{align*}
 by Claim \ref{cla:BridgingDifference}, where the constant is $100 + 4(c_\S)^2 + 14 c_\S$, since $\sqrt K \leq K$ for all $K \geq 1$.
  
Lastly, we wish to relate the number $|\Delta_i|$ to the length of boundary subword $b_i$. Let $p(\gamma)$ be the path formed by concatenating letters in $w_\Pi(\gamma)$. Lift the homotopy between $\gamma$ and $p(\gamma)$ so that $\gamma$ lifts to $\tilde \gamma$. This gives us a lift $\tilde p(\gamma)$. Then there is a lift $\tilde b_i$ of $b_i$ that lies on $\tilde \beta_l$. Each edge in $\tilde b_i$ lies on the boundary of two hexagons. By construction, $\tilde \gamma$ must pass through at least one of those hexagons. So if $\tilde \gamma$ enters $N_1(\tilde \beta_l)$ at hexagon $h_{n_i}$ and exits at hexagon $\chi_{m_i}$, then $|b_i| \leq |m_i - n_i| + 1$. In other words,
\begin{equation}
\label{eq:biLessThanDelta}
 |b_i| \leq |\Delta_i| + 1
\end{equation}
Note that if $\tilde \gamma_i$ passes through hexagons numbered $n-1, n, n+1$ adjacent to $\tilde \beta_l$, then $\tilde p(\gamma)$ must lie on $\tilde \beta_l$ in hexagon number $n$ (at least). So the curve $\tilde p(\gamma)$ must lie on $\tilde \beta_l$ at least from hexagon $n_i +1$ to hexagon $m_i -1$, as in Figure \ref{fig:BoundBiBelowByDelta}.
  \begin{figure}[h!]
   \centering 
   \includegraphics{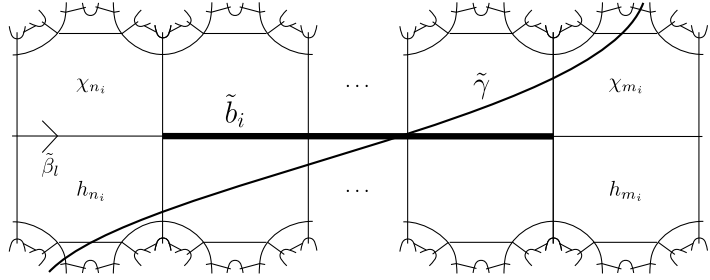}
   \caption{The arc $\tilde b_i$ must contain the arc in bold.}
   \label{fig:BoundBiBelowByDelta}
  \end{figure}
 Combining this observation with Inequality (\ref{eq:biLessThanDelta}), we get
  \begin{equation}
  \label{eq:biBoundedByDelta}
   |b_i| - 1 \leq |\Delta_i| \leq |b_i| + 1
  \end{equation}
   
   Let $B$ be the set of all bridging boundary subarcs, and let $B_l$ be the subarcs of $\gamma$ that bridge $\beta_l$. Then
   \begin{align*}
    \sum_{\gamma_i \in B} |b_i| & \leq \sum_{l=1}^m \sum_{\gamma_i \in   B_l} (|\Delta_i| + 1)\\
    & \leq  \sum_{l=1}^m \sum_{\gamma_i \in   B_l} |  \Delta_i| + c_\S \sqrt K\\
     & \lesssim K
   \end{align*}
  Note that the second inequality uses the fact that $\sum_{l=1}^m \# B_l \leq c_\S \sqrt K$. We can take the multiplicative constant to be $(100 + 29(c_s)^2)m$, since $\sqrt K \leq K$ for all $K \geq 1$ and $c_\S \geq 1$.

  \end{proof}
 
 Now we want to find the element $\gamma' \in \Mod_\S \cdot \gamma$ for which $i(\gamma', \Pi) \lesssim \sqrt K$ and which also satisfies inequality (\ref{eq:DeltaBalanced}). To do this, we will apply Dehn twists to $\gamma$ about curves in $\Pi$ until the result satisfies (\ref{eq:DeltaBalanced}). Recall that $\tau_l$ was the Dehn twist about $\beta_l \in \Pi$ for each $l$.

  \begin{cla}
  \label{cla:ApplyDehnTwists}
   There is a composition of Dehn twists $f = \tau_1^{n_1} \dots \tau_m^{n_m}$ so that the pair $(f \cdot \gamma, \Pi)$ satisfies (\ref{eq:DeltaBalanced}) for each $\beta_l \in \Pi$:
   \[
   0 \leq \sum_{\gamma_i \in B_l} \Delta_i  \leq 4 c_\S \sqrt K
   \]
  where $B_l$ is the set of boundary subarcs of the pair $(f \cdot \gamma, \Pi)$ that bridge $\beta_l$, and $\Delta_i$ is the twisting parameter of $\gamma_i \in B_l$.
  \end{cla}
  
  \begin{proof}
   Consider the closed curve $p(\gamma)$ formed by concatenating the letters of $w_\Pi(\gamma)$. If $b_i \in B_l$ is a boundary subword bridging $\beta_l$, then the orientation of $\S$ assigns $b_i$ a sign $\sigma = \pm 1$ depending on where $b_i$ twists around $\beta_l$ to the right or to the left. (For example, we can define twisting to the right to be positive twisting.) Assign each $b_i \in B$ the twisting parameter
   \[
    \sigma_i |b_i|
   \]
   Note that knowing the twisting parameter and the curve $\beta_l$ uniquely determines $b_i$, up to orientation.
  
  Let $f = (\tau_1)^{n_1} \cdots (\tau_m)^{n_m} \in \Mod_\S$ be a composition of Dehn twists about $\Pi$. Then a representative of $f \cdot p(\gamma)$ can be obtained by changing only the bridging boundary subwords of $p(\gamma)$. In fact, the bridging boundary subword $f \cdot b_i$ of $f \cdot p(\gamma)$ will be the unique one whose twisting parameter is
  \[
   \sigma_i|b_i| + 2 n_l
  \]
  
  Thus, twisting once about $\beta_l$ adds $2\#B_l$ to the sum $\sum_{B_l} \sigma_i |b_i|$. As $\# B_l \leq c_\S \sqrt K$, we can choose integers $n_1, \dots, n_m$ so that
  \begin{equation}
  \label{eq:BoundWordTwists}
   c_\S \sqrt K \leq \sum_{B_l} \large ( \sigma_i |b_i| + 2 n_l \large ) \leq 3 c_\S \sqrt K
  \end{equation}
  for each $l$.  This means that the net twisting of $f \cdot p(\gamma)$ about $\beta_l$ is bounded in terms of $\sqrt K$.
  
  We wish to relate the net twisting of $f \cdot p(\gamma)$ with the net twisting of $f \cdot \gamma$. To do this, we will first show that the closed curve $f \cdot p(\gamma)$ is, in fact, the curve formed by concatenating the letters of $w_\Pi(f \cdot \gamma)$.
  
  Let $w$ be the cyclic word given by the edges in $f \cdot p(\gamma)$. When we went from $p(\gamma)$ to $f \cdot p(\gamma)$, we only changed the length and direction of bridging boundary subwords.  Thus, $w$ satisfies Definition \ref{defi:AllowableWords_Surface} because $w_\Pi(\gamma)$ does. In particular, $w$ is an allowable word living in $W_\Pi$. Furthermore, $f \cdot p(\gamma)$ must have the least possible number of boundary subwords among all such curves freely homotopic to $f \cdot \gamma$, because it has the same number of boundary subwords as $p(\gamma)$, which had the fewest possible number of boundary subwords. Therefore, $w = w_\Pi(f \cdot \gamma)$.
  
  Take a word $w = w_\Pi(\gamma)$ for some closed geodesic $\gamma$. Suppose $b_i$ is a bridging boundary subword and $\gamma_i$ is the corresponding boundary subarc for the pair $(\gamma, \Pi)$. We need to relate the twisting parameter $\Delta_i$ of $\gamma_i$ to the twisting parameter $\sigma_i |b_i|$ of $b_i$. Note that the sign of $\Delta_i$ is the same as the sign $\sigma_i$ of $b_i$ (if we let 0 have whatever sign is needed for this statement to hold). Thus, Inequality (\ref{eq:biLessThanDelta}) implies that 
  \[
   \sigma_i|b_i| - 1 \leq \Delta_i \leq \sigma_i|b_i| + 1 
  \]
  
  Let $f \cdot B_l = \{f \cdot \Delta_i\}$ be the set of twisting parameters of the bridging boundary subarcs of the pair $(f \cdot \gamma, \Pi)$ that bridge $\beta_l$. Then by Inequality \ref{eq:BoundWordTwists}, 
  \[
   0 \leq \sum_{f \cdot B_l} f \cdot \Delta_i \leq 4 c_\S \sqrt K
  \]
  \end{proof}
      
    \subsection{Proof of Lemma \ref{lem:BridgingBdrySubwordIntBound}}
    \label{sec:BridgingCombined}  
    Take the curve $\gamma' = f \cdot \gamma \in \Mod_\S \cdot \gamma$ defined in the previous claim. Note that $i(\gamma', \Pi) = i(\gamma, \Pi)$, since $\gamma'$ differs from $\gamma$ by Dehn twists about curves in $\Pi$. Furthermore, $\gamma'$ satisfies (\ref{eq:DeltaBalanced}). So Claims \ref{cla:BridgingDifference} and \ref{cla:BridgingTotalLength} imply that if $w_\Pi(\gamma') = b_1s_1 \dots b_ns_n$, and if $B$ is the set of bridging boundary subarcs of the pair $(\gamma', \Pi)$, then
    \[
     \sum_{\gamma_i \in B} |b_i| \lesssim i(\gamma', \gamma')
    \]
    Since $i(\gamma', \gamma') = i(\gamma, \gamma)$, we are done.
    \end{proof}

\section{Proposition \ref{prop:BdrySubwordIntBound} for interior boundary subwords}
\label{sec:InteriorBdrySubwordIntBound}

Next, we show Proposition \ref{prop:BdrySubwordIntBound} for interior boundary subwords:

\begin{lem}
\label{lem:InteriorBdrySubwordIntBound}
 Let $L,K > 0$. Let $\gamma \in \G^c(L,K)$ and let $\Pi$ be any pants decomposition so that $i(\gamma, \Pi) \leq c_\S \sqrt K$. If $w_{\Pi}(\gamma) = b_1s_1 \dots b_n s_n$ then
 \begin{equation}
 \label{eq1}
  \sum_{\Gamma'} |b_i| \lesssim K
 \end{equation}
 and furthermore,
 \begin{equation}
 \label{eq2}
  \#\Gamma' \lesssim \sqrt K
 \end{equation}
where the constants depend only on the topology of $\S$ and $\Gamma'$ is the set of interior boundary subarcs of the pair $(\gamma, \Pi)$.
\end{lem}

For the rest of this section, fix a $\gamma \in \G^c(L,K)$ and choose a pants decomposition $\Pi$ so that $i(\gamma, \Pi) \leq c_\S \sqrt K$. Write
\[
 w_\Pi(\gamma) = b_1s_1 \dots b_n s_n
\]
Then $\gamma_i = \alpha(b_i)$ will be the boundary subarc associated to $b_i$ and the pair $(\gamma, \Pi)$.

\subsection{Intuition behind Lemma \ref{lem:InteriorBdrySubwordIntBound}}
\label{sec:InteriorBdrySubwordIntuition}
  The next part is intended to provide intuition for why Lemma \ref{lem:InteriorBdrySubwordIntBound} holds, as the proof itself is rather technical. The full outline of the proof can be found in Section \ref{sec:IdeaLemmaInteriorBdry}.
For now, suppose that $\gamma$ lies in a single pair of pants, so that all boundary subwords are interior boundary subwords. Suppose $b_i$ is an interior boundary subword, and $\gamma_i$ is the corresponding boundary subarc. Intuitively, if $|b_i| = 2n$ then 
\[
i(\gamma_i,\gamma_i) \approx n	
\]
We will save the explanation for later, but refer to Figure \ref{fig:BoundarySubarcs} on page \pageref{fig:BoundarySubarcs} for an illustration of why this should be true. In Figure \ref{fig:BoundarySubarcs}, the interior boundary subarc $b_i$ has length 6, and $\gamma_i$ has 3 self-intersections.

Moreover, we note that if two interior boundary subarcs, $b_i$ and $b_j$, wind around the same curve in $\Pi$, they will interfere with one another. Roughly, if $|b_i| = 2n$ and $|b_j| = 2m$, then 
\[
i(\gamma_i, \gamma_j) \approx \min \{2n, 2m\}
\]
(Figure \ref{fig:OverlapTies}.) Thus we can approximate the number of intersections between boundary subarcs if we just know the lengths of the corresponding boundary subwords.

 \begin{figure}[h!] \centering
  \includegraphics{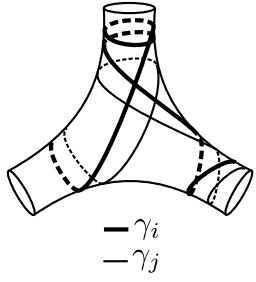}
 \caption[Relationship between boundary subword length and intersection number]{$|b_i| = 4, |b_j| = 2$, and $i(\gamma_i, \gamma_j) = 2$}
 \label{fig:OverlapTies}
 \end{figure}

So it is not unreasonable to assume that we have the following lower bound
\begin{equation}
\label{eq3}
i(\gamma_i, \gamma_j) \gtrsim \min \{|b_i|, |b_j|\}
\end{equation}
for all pairs of interior boundary subwords $(b_i, b_j)$ and some universal constant. 
We would then get Proposition \ref{prop:BdrySubwordIntBound} by summing inequality (\ref{eq3}) over all $i \leq j$. 
Specifically, relabel the interior boundary subwords so that $|b_1| \geq |b_2| \geq \dots \geq |b_n|$. Then, $\min \{|b_i|, |b_j|\} = |b_i|$ if $i \leq j$. Thus,
\[
 \sum_{i\leq j = 1}^n i(\gamma_i, \gamma_j)  \gtrsim \sum_{i=1}^n i |b_i|
\]
as $|b_i|$ is less than or equal to $i$ other word lengths.

So, if inequality (\ref{eq3}) held for all pairs of interior boundary subwords, then Lemma \ref{lem:OverlapBound} implies that
\begin{equation}
\label{eq4}
\sum_{i=1}^n i |b_i| \lesssim i(\gamma, \gamma)
\end{equation}
(This is almost true: see Remark \ref{rem:Sumibi}.)

On the one hand, if we use that $|b_i| \leq i |b_i|$, inequality (\ref{eq4}) implies that
\[
\sum_{i=1}^n |b_i| \lesssim i(\gamma,\gamma)
\]
so the total length of all interior boundary subwords is coarsely bounded by intersection number.

On the other hand, we can use that $|b_i| \geq 2$ for all interior boundary subwords. Then inequality (\ref{eq4}) implies that
\begin{align*}
 n(n+1) & \lesssim i(\gamma, \gamma) \implies\\
 n^2 & \lesssim i(\gamma, \gamma)
\end{align*}
so the number of interior boundary subwords is coarsely bounded by the square root of intersection number. So if all pairs $(\gamma_i, \gamma_j)$ satisfied Inequality (\ref{eq3}), we would get Proposition \ref{prop:BdrySubwordIntBound}.

In general, not all pairs of interior boundary subwords satisfy inequality (\ref{eq3}). To deal with this, we first need some more definitions.

\subsection{Relevant subarcs, and an outline of the proof of Lemma \ref{lem:InteriorBdrySubwordIntBound}}
\label{sec:ProofEq3}

In Section \ref{sec:InteriorBdrySubwordIntuition}, we hoped to convince the reader that we want pairs of boundary subarcs $\gamma_i = \alpha(b_i)$ and $\gamma_j = \alpha(b_j)$ to satisfy Inequality (\ref{eq3}). 

So we will now investigate which pairs $(\gamma_i, \gamma_j)$ of interior boundary subarcs satisfy a more precise version of Inequality (\ref{eq3}). For that, we make the following definition:
\begin{defi}
\label{def:Relevant}
We say that a set $\Gamma$ of interior boundary subarcs is \textbf{relevant} if $\forall \gamma_i, \gamma_j \in \Gamma$,
\begin{equation}
\label{eq:3}
\frac 15 \min\{|b_i|, b_j|\} \leq i(\gamma_i, \gamma_j)
\end{equation}
where $\gamma_i = \alpha(b_i)$ and $\gamma_j = \alpha(b_j)$.
\end{defi}
Thus a relevant set of interior boundary subarcs is one in which any two elements satisfy inequality (\ref{eq:3}).

Suppose $\Pi = \{\beta_1, \dots, \beta_m\}$. Recall that $\Gamma'_{2j}$ and $\Gamma'_{2j+1}$ are the sets of interior boundary subarcs that lie on the positive and negative sides of $\beta_j$, respectively. Then there is no hope that a pair $(\gamma_i, \gamma_j)$ will satisfy inequality (\ref{eq:3}) if $\gamma_i$ and $\gamma_j$ lie in different sets $\Gamma'_k$ and $\Gamma'_l$, respectively. So, we wish to find a maximal relevant subset 
\[
\Gamma_i \subset \Gamma'_i
\]
for each $i$. 

Building the maximal relevant subset $\Gamma_i$ becomes easier if we put a restriction on which interior boundary subarcs we consider.

\begin{defi}
\label{def:BetaRelevant}
 Let $\gamma_i \in \Gamma'_{2j} \cup \Gamma'_{2j+1}$. Suppose we have a lift $\tilde \gamma_i = \tilde \gamma \cap N_k(\tilde \beta_j)$ of $\gamma_i$, as in Definition \ref{def:TwistTie}. Then $\gamma_i$ is \textbf{twisting} if
 \[
  f(\tilde \gamma_i) \cap \tilde \gamma_i \neq \emptyset
 \]
where $f$ is a deck transformation of $\tilde \S$ acting by translation along $\tilde \beta_j$.
\end{defi}
Essentially, $\gamma_i \in \Gamma'_{2j}$ or $\Gamma'_{2j+1}$ is twisting if at least one of its self-intersections comes from twisting around $\beta_j$.

\subsubsection{Idea of proof of Lemma \ref{lem:InteriorBdrySubwordIntBound}}
\label{sec:IdeaLemmaInteriorBdry}
Then we show Lemma \ref{lem:InteriorBdrySubwordIntBound} as follows:
 \begin{itemize}
 \item We show that there is a unique maximal relevant subset $\Gamma_j \subset \Gamma'_j$, for each $j$, that consists entirely of twisting subarcs. In particular, we show that if $\Gamma_1, \Gamma_2 \subset \Gamma'_i$ are relevant subsets, and if each element of $\Gamma_1 \cup \Gamma_2$ is twisting, then $\Gamma_1 \cup \Gamma_2$ is a relevant subset of $\Gamma'_i$ (Lemma \ref{lem:UnionRelevant}).
 
 \item Let $\Gamma = \cup \Gamma_j$ be the union of all the maximal relevant subsets. Because all pairs of interior boundary subarcs in each $\Gamma_j$ satisfy inequality (\ref{eq:3}), a computation similar to the one in Section \ref{sec:InteriorBdrySubwordIntuition} implies that
 \[
  \sum_{\gamma_i \in \Gamma} |b_i| \lesssim K
 \]
 and
 \[
  \# \Gamma \lesssim \sqrt K
 \]
 for each $j$ (Section \ref{sec:IntLemForRelevantSubsets}).
 
 \item The union of all maximal relevant subsets turns out to be quite large. We show that
 \[
  3 \# \Gamma + 4 \# B \geq \# \Gamma'
 \]
where $\Gamma = \cup \Gamma_j$, and $B$ is the set of all bridging boundary subarcs (Lemma \ref{lem:RelevantAndBridgingPercent}.) 

We know that $\#B \leq c_\S \sqrt K$. So this allows us to promote the above inequalities for $\Gamma$ to inequalities for all of $\Gamma'$. Namely, we show that $\sum_{\Gamma'} |b_i| \lesssim K$ and $\# \Gamma' \lesssim \sqrt K$, completing the proof of Lemma \ref{lem:InteriorBdrySubwordIntBound} (Section \ref{sec:ProofInteriorBdrySubwordIntBound}).
\end{itemize}

\subsection{Finding a maximal relevant subset}

The following lemma implies that each $\Gamma_j'$ contains a maximal relevant subset $\Gamma_j$, where every element is twisting.
\begin{lem}
\label{lem:UnionRelevant}
 Let $\Gamma_1, \Gamma_2 \subset \Gamma'_k$ be relevant subsets so that each element of $\Gamma_1 \cup \Gamma_2$ is twisting. Then $\Gamma_1 \cup \Gamma_2$ is a relevant subset of $\Gamma'_k$.
\end{lem}

\begin{proof}
 To prove this lemma, we need to take $\gamma_i= \alpha(b_i) \in \Gamma_1, \gamma_j = \alpha(b_j) \in \Gamma_2$ and show that 
\[
 \frac 15 \min\{|b_i|, |b_j|\} \leq i(\gamma_i, \gamma_j)
\]

Suppose $\Gamma'_k$ is the set of interior boundary subarcs that lie on some side of $\beta_l \in \Pi$. Lift the hexagon decomposition of $\S$ to the universal cover, and take some lift $\tilde \beta_l$ of $\beta_l$. 

Take two lifts $\tilde \gamma$ and $\tilde \gamma'$ of $\gamma$ so that we get lifts $\tilde \gamma_i$ and $\tilde \gamma_j$ of $\gamma_i$ and $\gamma_j$, respectively, with
\[
 \tilde \gamma_i = \tilde \gamma \cap N_*(\tilde \beta_l) \mbox{ and } \tilde \gamma_j = \tilde \gamma' \cap N_*(\tilde \beta_l)
\]
where $*$ is either 1 or 2. 

It is easier to show that the pair $(b_i, b_j)$ satisfies inequality (\ref{eq:3}) if $|b_i|, |b_j| \geq 4$. So we do that first:
\begin{cla}
\label{cla:Length4Relevant}
 Let $\gamma_i = \alpha(b_i), \gamma_j = \alpha(b_j) \in \Gamma'_k$. Suppose $|b_i|, |b_j| \geq 4$. Then 
 \[
  \frac 15 \min\{|b_i|, b_j|\} \leq i(\gamma_i, \gamma_j)
 \]
\end{cla}
\begin{proof}
Number the hexagons on either side of $\tilde \beta_l$ as in the proof of Lemma \ref{lem:BridgingTotalLength} (Figure \ref{fig:HexagonOrder}). Without loss of generality, $\tilde \gamma$ and $\tilde \gamma'$ lie on the side of $\tilde \beta_l$ with hexagons labeled $\dots, h_0, h_1, \dots$. Suppose $\tilde \gamma_i$ enters $N_1(\tilde \beta_l)$ at hexagon $h_{n_i}$ and exits at hexagon $h_{m_i}$, and likewise, $\tilde \gamma_j$ enters at hexagon $h_{n_j}$ and exits at $h_{m_j}$. Up to changing orientation of $\tilde \gamma$ and $\tilde \gamma'$, we can assume that $n_i < m_i$ and $n_j < m_j$. Then $\tilde \gamma_i$ intersects $\tilde \gamma_j$ if
\begin{align*}
 n_i & < n_j < m_i < m_j \mbox{ or } \\
 n_j & < n_i < m_j < m_i
\end{align*}
(Figure \ref{fig:Delta_iDelta_j}).

\begin{figure}[h!]
 \centering 
 \includegraphics{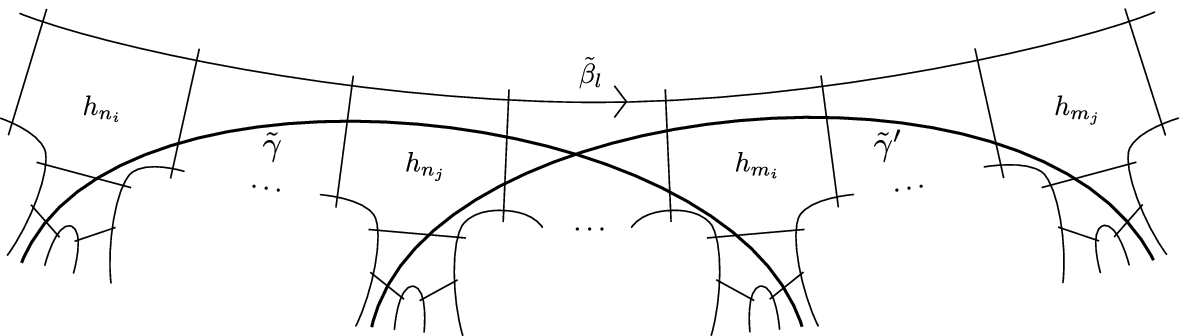}
 \caption{}
 \label{fig:Delta_iDelta_j}
\end{figure}

\begin{rem}
\label{rem:IntAndBridgeHexNumbers}
 Suppose instead of two interior boundary subarcs, we consider $\gamma_i$ as above, but let $\gamma_j$ be a boundary subarc bridging $\beta_l$. Then we get a similar inequality. In particular, if a lift $\tilde \gamma_j$ of $\gamma_j$ intersects $\tilde \beta_l$, then $\# \tilde \gamma_i \cap \tilde \gamma_j = 1$ if $n_i \leq n_j \leq m_i$.
\end{rem}

Just as in the proof of Lemma \ref{lem:BridgingTotalLength}, let $f_l$ be the deck transformation that acts by translation along $\tilde \beta_l$ with translation length $l(\beta_l)$. We will use $f_l$ to slide $\tilde \gamma_i$ around and create intersections with $\tilde \gamma_j$.

Up to replacing $f_l$ with $f_l^{-1}$, we have $f_l(h_n) = h_{n+2}$. So $f^k_l(\tilde \gamma_i)$ enters and exits $N_1(\tilde \beta_l)$ in hexagons $h_{n_i + 2k}$ and $h_{m_i + 2k}$, respectively. Thus, whenever
\begin{align*}
 n_{i+2k} < n_j &  < m_{i+2k} < m_j \mbox{ or }  \\
 n_j < n_{i+2k} & <  m_j < m_{i+2k}
\end{align*}
we have that
\[
 \# f^k_l (\tilde \gamma_i) \cap \tilde \gamma_j = 1
\]

So we can count the number of powers $k$ that result in an intersection:
\[
 \#\{ k \in \Z \ | \ \# f_l^k (\tilde \gamma_i) \cap \tilde \gamma_j = 1 \} \geq \min \{m_i - n_i, m_j - n_j\} - 2
\]

There are two things left to do:
\begin{enumerate}
 \item Show that the set of powers $k$ that result in intersections between $\tilde \gamma_i$ and $\tilde \gamma_j$ in $\tilde \S$ are in at most 2-to-1 correspondence with intersections of $\gamma_i$ and $\gamma_j$.
 \item Show that $m_i - n_i \geq |b_i| -1 $.
\end{enumerate}

Let 
\[
 \{\tilde x_k\} = f_l^k(\tilde \gamma_i) \cap \tilde \gamma_j 
\]
for each power $k$ so that $f_l^k(\tilde \gamma_i) \cap \tilde \gamma_j  \neq \emptyset$. If $k \neq k'$, then $\tilde x_k \neq \tilde x_{k'}$. Let $\pi(\tilde x_k) = x_k$ be the projection of the intersection points to $\S$, for each $k$. The problem is if $x_k = x_{k'}$ for some $k$ and $k'$.  But $\tilde x_k$ and $\tilde x_{k'}$ project to the same point in $\S$ if and only if they correspond to the same self-intersection point of $\gamma$. By Lemma \ref{lem:PiHomeo}, $l(\gamma_i) \leq l(\gamma)$, so $\gamma_i$ and $\gamma_j$ are proper subarcs of $\gamma$. So in fact, for each $\tilde x_k$ there is at most one other $\tilde x_{k'}$ that projects down to the same intersection between $\gamma_i$ and $\gamma_j$. (And, in fact, if $i = j$, each $\tilde x_k$ gets paired up with exactly one other $\tilde x_{k'}$ in this way.)

Thus, the set
\[
 \{ k \in \Z \ | \ \# f_l^k (\tilde \gamma_i) \cap \tilde \gamma_j = 1 \}
\]
is in at most 2-to-1 correspondence with intersections between $\gamma_i$ and $\gamma_j$. Therefore, 
\[
i(\gamma_i, \gamma_j) \geq \frac 12 \#\{ k \in \Z \ | \ \# f_l^k (\tilde \gamma_i) \cap \tilde \gamma_j = 1 \}
\]
and so,
\[
 i(\gamma_i, \gamma_j) \geq \frac 12 \min \{m_i - n_i, m_j - n_j\} - 1
\]

 Lastly, we will show that $m_i - n_i \geq |b_i| -1$. In fact, we can use the same argument as in the proof of Claim \ref{cla:BridgingTotalLength}. Let $p(\gamma)$ be the curve formed by concatenating the edges in $w(\gamma)$. There is a homotopy between $p(\gamma)$ and $\gamma$. We can lift this homotopy so that $\gamma$ lifts to $\tilde \gamma$ and $p(\gamma)$ lifts to $\tilde p(\gamma)$. Each edge of $\tilde p(\gamma)$ lies on the boundary of two hexagons. If an edge $e$ lies on the boundary of $h$ and $h'$, then by construction, $\tilde \gamma$ must pass through either $h$ or $h'$. So $\tilde \gamma$ must pass through at least $|b_i|$ hexagons in $N_1(\tilde \beta_i)$. But $\gamma_i$ passes through exactly $m_i - n_i + 1$ hexagons in $N_1(\tilde \beta)$. Thus,
 \[
  m_i - n_i \geq |b_i| - 1 \mbox{ and } m_j - n_j \geq |b_j| -1
 \]
 Therefore,
 \begin{align*}
i(\gamma_i, \gamma_j) & \geq \frac 12 \min \{|b_i|, |b_j|\} \} - \frac 32 \\
                      & = \min \frac 12 \{|b_i| - 3, |b_j| - 3\}
\end{align*}
 As $|b_i|, |b_j| \geq 4$, this implies
 \[
  i(\gamma_i, \gamma_j) \geq \frac 15 \min \{|b_i|, |b_j|\} \}
 \]

\end{proof}

We showed that if $|b_i| \geq 4$ then $\{\gamma_i\}$ is relevant. In fact, the proof of Claim \ref{cla:Length4Relevant} also implies the following.
\begin{cor}
\label{rem:Length4Relevant}
 Let $b_i$ be an interior boundary subword. If $|b_i| \geq 4$ then the boundary subarc $\gamma_i$ is twisting.
\end{cor}

Moreover, an even stronger statement follows from the proof. Note that $|b_i| \geq 4$ implies that $\tilde \gamma_i$  passes through at least 4 hexagons in $N_1(\tilde \beta_l)$, but not vice versa.
\begin{cor}
\label{rem:Pass4Hexagons}
 Let $b_i$ be an interior boundary subword. If $\tilde \gamma_i$ passes through at least 4 hexagons in $N_1(\tilde \beta_l)$ then $\{\gamma_i\}$ is relevant and $\gamma_i$ is twisting.
\end{cor}

\noindent \textit{Completion of proof of Lemma \ref{lem:UnionRelevant}.}
 Let $\Gamma_1, \Gamma_2 \subset \Gamma'_k$ be two relevant subsets so that each element of $\Gamma_1 \cup \Gamma_2$ is twisting. We need to show that given any $\gamma_i = \alpha(b_i) \in \Gamma_1$ and $\gamma_j = \alpha(b_j) \in \Gamma_2$ that
 \[
  i(\gamma_i, \gamma_j) \geq \frac 15 \min\{|b_i|, |b_j|\}
 \]

 If $|b_i|, |b_j| \geq 4$, then we are done by Claim \ref{cla:Length4Relevant}. So suppose that $\min \{ |b_i|, |b_j| \} \leq 3$. Thus, $\frac 15 \min\{|b_i|, |b_j|\} \leq 1$. So we just need to show that $i(\gamma_i, \gamma_j) \geq 1$.
 
 Again, consider the lifts $\tilde \gamma_i$ and $\tilde \gamma_j$, as above. Since $\gamma_i$ and $\gamma_j$ are twisting, we have the deck transformation $f$ acting by translation along $\tilde \beta_l$ so that
 \[
  \tilde \gamma_i \cap f \cdot \tilde \gamma_i \neq \emptyset \mbox{ and } \tilde \gamma_j \cap f \cdot \tilde \gamma_j \neq \emptyset
 \]
 Take the set $\{f^m \cdot \tilde \gamma_i\}_{m \in \Z}$ of all translations of $\tilde \gamma_i$ by $f$. Because $f^m \cdot \tilde \gamma_i$ intersects $f^{m+1} \cdot \tilde \gamma_i$ for each $m$, we get a region $R_i$ whose boundary consists of $\tilde \beta_l$ and a subarc of $\cup_m f^m(\gamma_i)$. Likewise, we can form a region $R_j$ from the translates of $\tilde \gamma_j$ that has the same property (Figure \ref{fig:CrossRegion}.) 
 \begin{figure}[h!]
  \centering
  \includegraphics{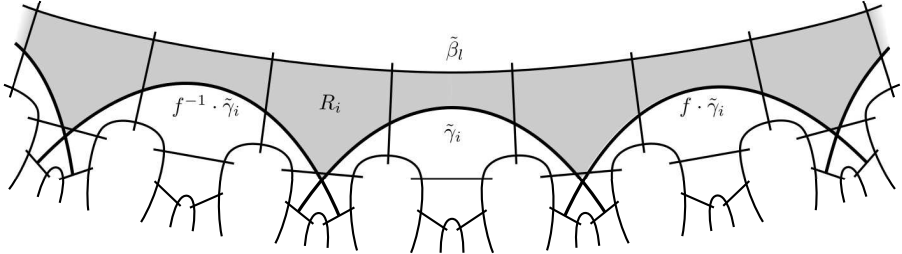}
  \caption{The region $R_i$ is shaded.}
  \label{fig:CrossRegion}
 \end{figure}

 These regions overlap in a neighborhood of $\tilde \beta_l$. So there are two cases. Either one of $R_i$ and $R_j$ contains the other, or neither $R_i$ nor $R_j$ contains the other. First, suppose without loss of generality that 
 \[
  R_i \subset R_j
 \]
 We know that $\tilde \gamma_i$ has endpoints on $\partial N_*(\tilde \beta_l)$ where $* = 1$ or 2, and passes along the boundary of $R_i$, which is contained inside $R_j$. Furthermore, $R_j$ is a proper subset of $N_*(\tilde \beta_l)$. Thus, $\tilde \gamma_i$ must pass through the boundary of $R_j$. Since $\tilde \gamma_i$ does not cross $\tilde \beta_l$, there is some $m$ so that $\tilde \gamma_i \cap f^m(\tilde \gamma_j) \neq \emptyset$. In other words, $i(\gamma_i, \gamma_j) \geq 1$.
 
 Suppose now that
 \[
  R_i \not \subset R_j \mbox{ and } R_j \not \subset R_i
 \]
 Then the boundary of $R_i$ must intersect the boundary of $R_j$ somewhere. So there are powers $m$ and $m'$ so that $f^m \cdot \tilde \gamma_i  \cap f^{m'} \cdot \tilde \gamma_j \neq \emptyset$. Again, we conclude that $i(\gamma_i, \gamma_j) \geq 1$.
\end{proof}

 \begin{rem}
  Let $b_i \in \Gamma'_{2k} \cup \Gamma'_{2k+1}$ be a twisting interior boundary subword, and let $b_j \in B_k$ be a bridging boundary subword. Then the above proof also estimates the contribution that the pair $(b_i, b_j)$ make to intersection number. In particular,
  \[
   i(\gamma_i, \gamma_j) \geq \frac 15 |b_i| 
  \]
  where we use Remark \ref{rem:IntAndBridgeHexNumbers} for the case where $|b_i| \geq 4$.
  \end{rem}

\subsection{Lemma \ref{lem:InteriorBdrySubwordIntBound} for relevant subsets}
\label{sec:IntLemForRelevantSubsets}
By Lemma \ref{lem:UnionRelevant}, each set $\Gamma_k'$ of interior boundary subarcs has a unique maximal relevant subset $\Gamma_k \subset \Gamma_k'$ consisting entirely of twisting boundary subarcs. Let 
\[
\Gamma = \cup \Gamma_k
\]
 be the union of these maximal relevant subsets. Then the intuitive argument for Lemma \ref{lem:InteriorBdrySubwordIntBound} in Section \ref{sec:InteriorBdrySubwordIntuition} gives us the inequalities for $\Gamma$. 
\begin{lem}
\label{lem:MaxRelevantBdrySubwordIntBound}
 For any pair $(\gamma, \Pi)$, we have that
 \[
  \sum_{\Gamma} |b_i| \lesssim K
 \]
 and 
 \[
  \# \Gamma \lesssim \sqrt K
 \]
 where $\Gamma$ is the union of maximal relevant subsets of interior boundary subarcs for $(\gamma, \Pi)$, and the constants depend only on $\S$.
\end{lem}

\begin{proof}
 We actually show a stronger inequality for each maximal relevant subset $\Gamma_k$. We show that if we renumber the elements of $\Gamma_k$ by $\gamma_1, \dots, \gamma_{n_k}$ so that $|b_1| \geq |b_2| \geq \dots \geq |b_{n_k}|$, then
 \[
  \sum_{i=1}^{n_k} i|b_i| \lesssim K
 \]
 
 In fact, for all pairs $(b_i, b_j)$, where $b_i, b_j \in \Gamma_k$, we have that
 \[
  i(\gamma_i, \gamma_j) \geq \frac 15 \min\{b_i, b_j\}
 \]
 Therefore, if we relabel the elements of $\Gamma_k$ so that $|b_1| \geq |b_2| \geq \dots \geq |b_{n_k}|$, then
 \begin{align*}
  \sum_{i=1}^{n_k} i |b_i| & = \sum_{i \leq j} \min \{ |b_i|, |b_j|\} \\
   & \leq 25 \sum_{i \leq j} i(\gamma_i, \gamma_j) \\
   & \leq 625 i(\gamma, \gamma)
 \end{align*}
 where the last inequality is by Lemma \ref{lem:OverlapBound}.
 
 So if we use that $i|b_i| \geq |b_i|$, we get that
 \[
  \sum_{\Gamma_k} |b_i| \lesssim K
 \]
 for each $k$. Summing the above inequality over $\Gamma_1, \dots, \Gamma_{2m}$, we get that
 \[
  \sum_\Gamma |b_i| \lesssim K
 \]
 where the constant is $1250 m$.
 
 If we use instead that $|b_i| \geq 2$ for each interior boundary subarc, then we get $\sum_{\Gamma_k} 2i \lesssim K$. If we let
 \[
  N_k = \# \Gamma_k
 \]
 then this implies that $ N_k(N_k-1) \lesssim K$, so in particular,
 \[
  (N_k)^2 \lesssim K
 \]
 where the constant is 625. We can show by induction that $(\sum_{i=1}^n a_i)^2 \leq n \sum a_i^2$ for any sequence of numbers $a_1, \dots, a_n$. So let $N = \sum N_k$. That is, $N = \# \Gamma$. Then
 \[
  N^2 \leq 2m \sum_{k=1}^{2m}(N_k)^2
 \]
 Therefore,
 \[
  \# \Gamma \lesssim \sqrt K
 \]
 where the constant is $25 \sqrt{2m}$.
\end{proof}

\begin{rem}
\label{rem:Sumibi}
 We know that $\Gamma_k$ contains all $\gamma_i$ so that $|b_i| \geq 4$ by Corollary \ref{rem:Length4Relevant}. So the above proof actually gives us the following nice formula for $\Gamma'_k$:
 \[
  \sum_{|b_i| \geq 4, \gamma_i \in \Gamma'_k} i |b_i| \lesssim K
 \]
 where we number the elements of $\Gamma'_k$ so that $|b_1| \geq |b_2| \geq \dots \geq |b_{n_k}|$.
\end{rem}

\subsection{The maximal relevant subsets are large}
 Let $\Gamma$ be the union of maximal relevant subsets defined above. We want to show that $\Gamma$ is, in fact, quite large in the following sense.
\begin{lem}
 \label{lem:RelevantAndBridgingPercent}
Let $B$ and $\Gamma'$ be the set of bridging and interior boundary subarcs for the pair $(\gamma, \Pi)$, respectively. When $\Gamma \subset \Gamma'$ is defined as above, we have
  \[
   3\# \Gamma + 4\#B \geq  \#\Gamma'
  \]
 whenever the total number of boundary subarcs is at least 6.
 \end{lem}

 \begin{proof}

 Let $w_\Pi(\gamma) = b_1 s_1 \dots b_n s_n$ with $n \geq 6$. As usual, $\gamma_i = \alpha(b_i)$ is the boundary subarc associated to the boundary subword $b_i$. 
 
 \begin{cla}
 \label{cla:13Relevant}
If $\gamma_{i-2}, \gamma_{i-1}, \gamma_i, \gamma_{i+1}$ and $\gamma_{i+2}$ are all interior boundary subarcs, then one of $\{\gamma_{i-1}\}, \{\gamma_i\}$ and $\{\gamma_{i+1}\}$ must be relevant and twisting.  
 \end{cla}

 Given this claim, the proof of the lemma goes as follows: The union of twisting, relevant subsets of $\Gamma'_j$ is again twisting and relevant for each $j$. So, the set $\Gamma$ is exactly the union of all twisting and relevant singleton sets $\{\gamma_i\} \subset \Gamma'$. If we count such singleton sets, we get the size of $\Gamma$.
 
 Consider the (cyclic) sequence $(\gamma_1, \dots, \gamma_n)$ of boundary subarcs. We can break it up into  maximal subsequences of interior boundary subarcs: there are $\#B$ such sequences. Suppose after cyclic renumbering that $\gamma_1, \dots, \gamma_i$ is one such maximal sequence of interior boundary subarcs. Then Claim \ref{cla:13Relevant} implies that at least $\frac 13(i-4)$ of these boundary subarcs are relevant and twisting. So we get that
  \[
  \frac 13 (\# \Gamma' - 4 \# B) \leq \# \Gamma
 \]
This gives us Lemma \ref{lem:RelevantAndBridgingPercent} assuming Claim \ref{cla:13Relevant}.
\end{proof}
 
\begin{proof}[Proof of Claim \ref{cla:13Relevant}.]

We have that $\gamma_{i-2}, \dots, \gamma_{i+2}$ are all interior boundary subarcs. Thus, there is a pair of pants, $\p$, cut out by the pants decomposition $\Pi$, so that the subword $s_{i-3}b_{i-2} \dots b_{i+2} s_{i+2}$ is entirely contained in $\p$. That is, this subword doesn't cross $\partial \p$. Each boundary subarc $\gamma_j$ is defined in terms of the 2-hexagon neighborhood of a lift of the curve in $\Pi$ that contains $b_j$. Thus, the union $\gamma_{i-1} \cup \gamma_i \cup \gamma_{i+1}$ is entirely contained in $\p$, but $\gamma_{i+2}$ and $\gamma_{i-2}$ might cross $\partial \p$. This is why we only work with $\gamma_{i-1}, \gamma_i$ and $\gamma_{i+1}$ from now on.

Take a lift $\tilde \gamma$ of $\gamma$ to $\tilde \S$. Choosing a lift of the subarc $\gamma_{i-1} \cup \gamma_i \cup \gamma_{i+1}$ to $\tilde \gamma$ gives us lifts $\tilde \beta_j$ for $j=i-1,i,i+1$ of curves in $\Pi$ so that
\[
 \tilde \gamma_j = \tilde \gamma \cap N_2(\tilde \beta_j)
\]

By Corollary \ref{rem:Pass4Hexagons}, if $\tilde \gamma_i$ passes through four hexagons adjacent to $ \tilde \beta_i$, then $\gamma_i$ is twisting and the singleton set $\{\gamma_i\}$ is relevant. If this is the case, then we are done. So assume that $\tilde \gamma_i$ passes through at most three hexagons adjacent to $\tilde \beta_i$, as in Figure \ref{fig:Pass3}. 

In all the figures for this proof, we only draw the hexagon decomposition of the pair of pants $\p$ lifted to a partial hexagonal tiling of $\tilde \S$. This is because $\gamma_{i-1} \cup \gamma_i \cup \gamma_{i+1}$ lies in $\p$.

\begin{cla}
\label{cla:Pass3}
 The arc $\tilde \gamma_i$ can pass through at no fewer than 3 hexagons in $N_1(\tilde \beta_i)$.
\end{cla}
\begin{proof}

Because $b_i$ is an interior boundary subword, we know that $|b_i| \geq 2$. Thus, $\tilde \gamma_i$ passes through at least two hexagons in $N_1(\tilde \beta_i)$. Suppose that $|b_i| = 2$ and $\tilde \gamma_i$ passed through exactly 2 hexagons, $h_1$ and $h_2$ (Figure \ref{fig:Pass2}). Because $\gamma_i$ is an interior boundary subarc, $h_1$ and $h_2$ lie on the same side of $\tilde \beta_i$. Then the boundaries of $h_1$ and $h_2$ have exactly two lifts of curves in $\Pi$ in common. One is $\tilde \beta_i$ and the other is some curve $\tilde \beta_i'$. 

\begin{figure}[h!]
 \centering 
 \includegraphics{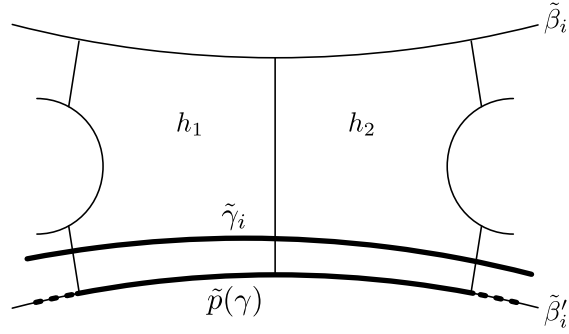}
 \caption{The case when $\tilde \gamma$ passes through just two hexagons in $N_1(\tilde \beta_i)$.}
 \label{fig:Pass2}
\end{figure}

Furthermore, $h_1$ and $h_2$ have three seam edges adjacent to $\tilde \beta_i'$. We have that $\tilde \gamma_i$ only crosses $N_1(\tilde \beta_i)$ in $h_1$ and $h_2$, and that it only crosses seam edges. So it must cross all three of the seam edges adjacent to $\tilde \beta_i'$. By the construction of $p(\gamma)$, this implies that $\tilde p(\gamma)$ must, in fact, lie on $\tilde \beta_i'$ and not $\tilde \beta_i$. This gives us a contradiction.

\end{proof}

By definition, $\tilde \gamma_i$ lives in $N_2(\tilde \beta_i)$. Label the hexagons it passes through in order. So their labels are $h_0, h_1, h_2, h_3, h_4$, where $h_1, h_2$ and $h_3$ lie in $N_1(\tilde \beta_i)$ and $h_0$ and $h_4$ do not. 
  
  \begin{figure}[h!]
   \centering 
   \includegraphics{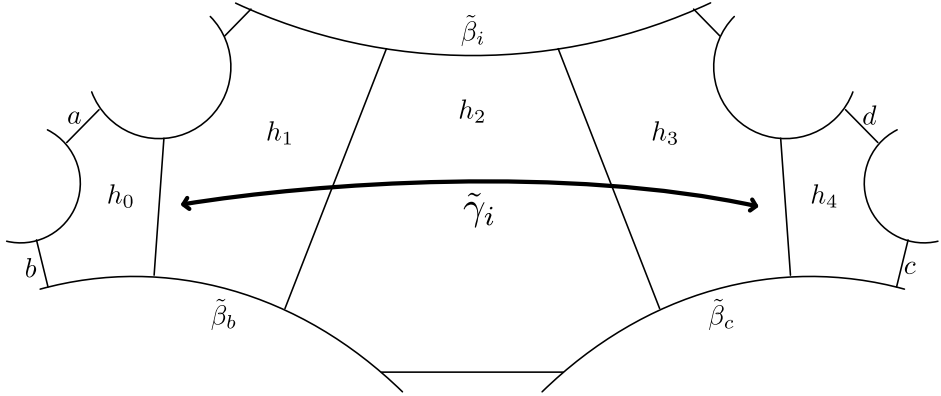}
   \caption{Part of $\tilde \S$ that contains $\tilde \gamma_i$.}
   \label{fig:Pass3}
  \end{figure}
  
Let $a$ and $b$ be the seam edges of $h_0$, and let $c$ and $d$ be the seam edges of $h_4$, that do not lie on $N_1(\tilde \beta)$ (Figure \ref{fig:Pass3}). Note that $b$ and $c$ are the seam edges adjacent to geodesics $\tilde \beta_b$ and $\tilde \beta_c$ that bound hexagon $h_2$. 

Because $\gamma_{i-1} \cup \gamma_i \cup \gamma_{i+1}$ lies on the interior of the pair of pants $\p$, the endpoints on $\gamma_i$ lie on seam edges. There are two cases: either the endpoints of $\tilde \gamma_i$ lie on $a$ and $d$, or at least one endpoint of $\tilde \gamma_i$ lies on $b$ or $c$.

Suppose $\tilde \gamma_i$ has an endpoint on $a$ and an endpoint on $d$. Consider the deck transformation $f$ that acts by translation along $\tilde \beta_i$ with translation length $l(\beta_i)$. Up to taking the inverse of $f$, we must have $f(h_1) = h_3$. So, $f(a) = c$.

Let $H = h_0 \cup \dots \cup h_4$ be the union of the five hexagons. If $e$ is the second seam edge that $\tilde \gamma_i$ hits as it goes through $H$, then $f(e)$ must be the seam edge on the boundary of $H$ as in Figure \ref{fig:Pass3ad}. Since $H$ is convex, and the pairs $a, d$ and $c, f(e)$ separate each other on $\partial H$, we must have
\[
 \tilde \gamma_i \cap f(\tilde \gamma_i) \neq \emptyset
\]
Therefore, $\gamma_i$ is twisting and $i(\gamma_i, \gamma_i) \geq 1$. As $|b_i| \leq 3$, this implies that $\{\gamma_i\}$ is relevant. 

  \begin{figure}[h!]
   \centering 
   \includegraphics{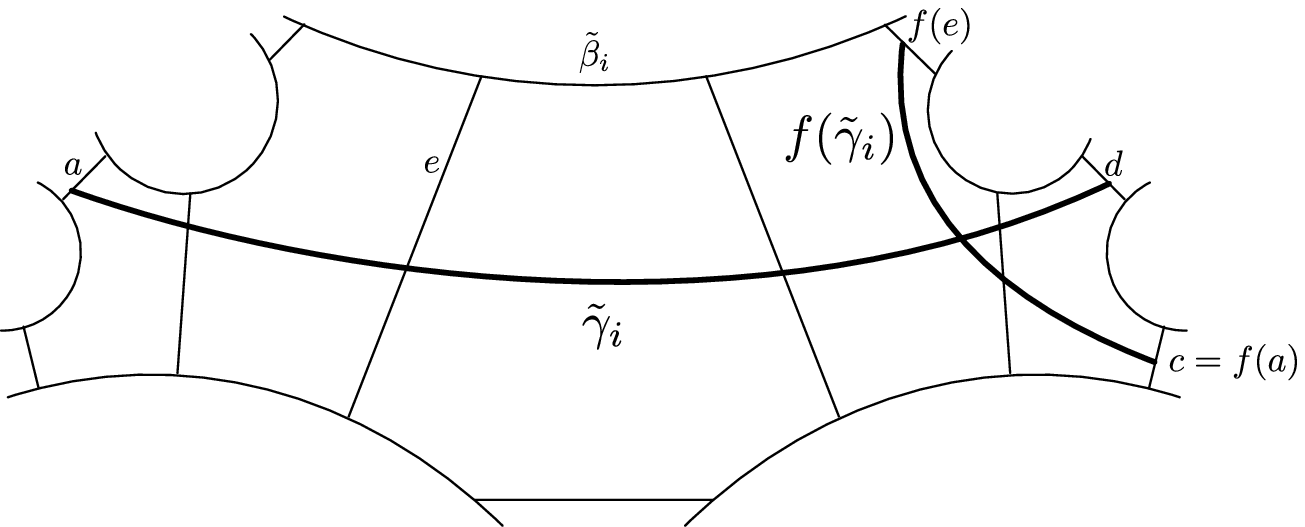}
   \caption{The case when $\tilde \gamma_i$ joins seams $a$ and $d$.}
   \label{fig:Pass3ad}
  \end{figure}

Now we consider the other case, where $\tilde \gamma_i$ has an endpoint on either $b$ or $c$ (Figure \ref{fig:Pass3bc}). Suppose without loss of generality that $\tilde \gamma_i$ has an endpoint on $b$. Thus, $\tilde \gamma$ continues past $b$ into another hexagon adjacent to $\tilde \beta_b$. So $\tilde \gamma$ must pass through at least 4 hexagons in $N_1(\tilde \beta_b)$.

\begin{figure}[h!]
 \centering 
 \includegraphics{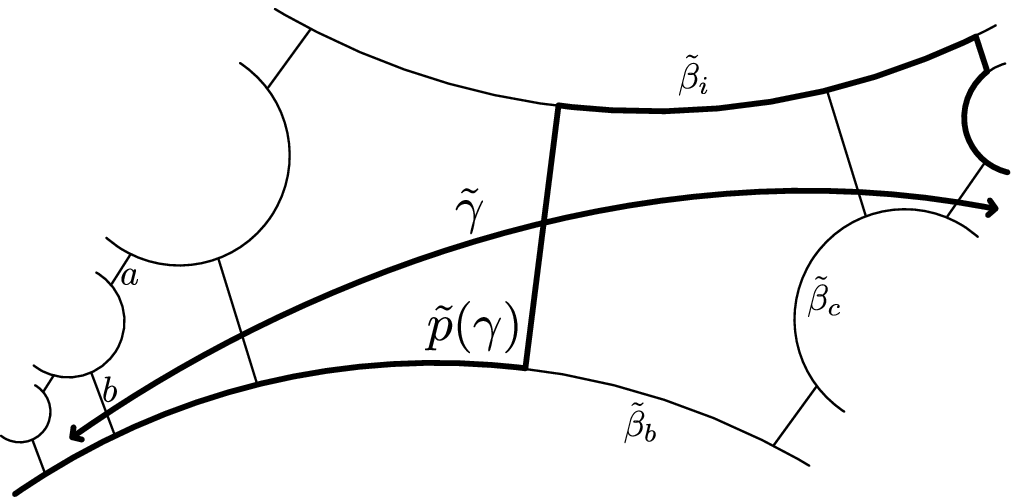}
 \caption{The case where $\tilde \gamma_i$ has an endpoint on seam $b$.}
 \label{fig:Pass3bc}
\end{figure}

Because $\tilde \gamma$ passes through four hexagons in $N_1(\tilde \beta_b)$, it must pass through at least three seam edges adjacent to $\tilde \beta_b$.
The construction of $p(\gamma)$ is such that if $\tilde \gamma$ passes through three seam edges adjacent to the lift $\tilde \beta_b$ of $\beta_b \in \Pi$, then $p(\gamma)$ has a boundary subword $b_j$ with a lift $\tilde b_j$ lying on $\tilde \beta_b$. Since $\tilde \beta_b$ and $\tilde \beta_i$ are joined by a seam edge, we must have $\tilde \beta_b = \tilde \beta_{i-1}$ or $\tilde \beta_{i+1}$. So, in fact, $\tilde \gamma$ passes through at least four hexagons adjacent to one of $\tilde \beta_{i-1}$ or $\tilde \beta_{i+1}$.  By Corollary \ref{rem:Length4Relevant}, this implies that either $\gamma_{i+1}$ or $\gamma_{i-1}$, respectively, is twisting and the singleton set containing this boundary subarc is relevant.

\end{proof}

\subsection{Proof of Lemma \ref{lem:InteriorBdrySubwordIntBound}}
\label{sec:ProofInteriorBdrySubwordIntBound}

\begin{proof}[Proof of Lemma \ref{lem:InteriorBdrySubwordIntBound}]
Recall that $\Gamma_k \subset \Gamma_k'$ is the maximal relevant subset defined above, for each $k = 1, \dots, 2m$, and that $\Gamma = \cup \Gamma_k$. By Lemma \ref{lem:MaxRelevantBdrySubwordIntBound}, we have that
\[
 \sum_{\Gamma}|b_i| \lesssim K \text{ and } \# \Gamma \lesssim \sqrt K
\]

First we use this to bound the size of $\Gamma'$ from below. If the pair $(\gamma, \Pi)$ has at least 6 boundary subwords, then Lemma \ref{lem:RelevantAndBridgingPercent} gives us that
  \[
   3\# \Gamma + 4\#B \geq  \#\Gamma'
  \]
We assumed that $i(\gamma, \Pi) \leq c_\S \sqrt K$. So by Proposition \ref{prop:BridgingIsCrossing}, $4 \# B \leq 4c_\S \sqrt K$. Since we have $\# \Gamma \lesssim \sqrt K$, this implies that
\[
 \# \Gamma' \lesssim \sqrt K
\]
where the constant is $25\sqrt{2m} + 4 c_\S$.

If the pair $(\gamma, \Pi)$ has has fewer than 6 boundary subwords, then $\# \Gamma' \leq 5$. As $K \geq 1$, this means $\# \Gamma' \lesssim \sqrt K$ with constant 5. So the above bound still holds.

Next we bound the total length of all interior boundary subwords from above. To do this, we need to bound the total length of all those interior boundary subwords whose corresponding subarcs are not in $\Gamma'$. By Corollary \ref{rem:Length4Relevant}, if $b_i$ is an interior boundary subword with $|b_i| \geq 4$, then $\gamma_i \in \Gamma$. So,
\begin{align*}
 \sum_{\Gamma' \setminus \Gamma} |b_i| & \leq \sum_{\Gamma' \setminus \Gamma} 4 \\
  &\lesssim \sqrt K
\end{align*}
where the last inequality comes from the fact that $\# \Gamma' \lesssim \sqrt K$, and the constant is $100\sqrt{2m} + 16c_\S$. As $\sqrt K \leq K$ for all $K \geq 1$, we can combine this inequality with Lemma \ref{lem:MaxRelevantBdrySubwordIntBound} to get that
\[
 \sum_{\Gamma'} |b_i| \lesssim K
\]
where the constant is $1450 m + 16c_\S$.
\end{proof}

\section{Proof of Proposition \ref{prop:BdrySubwordIntBound}}
\label{sec:ProofBdrySubwordIntBound}

Let $\gamma \in \G^c(L,K)$, for $L,K > 0$. By Lemma \ref{lem:BridgingBdrySubwordIntBound}, there is a $\gamma' \in \Mod_\S \cdot \gamma$ so that if $B$ is the set of bridging boundary subarcs of the pair $(\gamma', \Pi)$, then
\[
 \sum_{\gamma_i \in B} |b_i| \lesssim K \mbox{ and } \# B \leq c_\S \sqrt K
\]
Since $\# B \leq c_\S \sqrt K$, we have that $i(\gamma', \Pi) \leq c_\S \sqrt K$ by Proposition \ref{prop:BridgingIsCrossing}. So we can apply Lemma \ref{lem:InteriorBdrySubwordIntBound}. Thus, if $\Gamma'$ is the set of interior boundary subarcs of the pair $(\gamma', \Pi)$, then
\[
 \sum_{\gamma_i \in \Gamma'} |b_i| \lesssim K \mbox{ and } \# \Gamma' \lesssim \sqrt K
\]
The set of all boundary subarcs of the pair $(\gamma', \Pi)$ is exactly $B \cup \Gamma'$. Let $w_\Pi(\gamma') = b_1s_1 \dots b_n s_n$. In particular, there are $n$ boundary subwords. Therefore,
\[
 \sum_{i=1}^n |b_i| \lesssim K \mbox{ and } n \lesssim \sqrt K
\]
The right-hand inequality follows from the fact that $a^2 + b^2 \leq 2(a+b)^2$. 

So we have found a curve $\gamma' \in \Mod_\S \cdot \gamma$ and a pants decomposition $\Pi$ for which the inequalities in Proposition \ref{prop:BdrySubwordIntBound} hold. As explained in Section \ref{sec:PantsNotOnTheList}, this implies that there is some other $\gamma'' \in \Mod_\S \cdot \gamma$ and some pants decomposition $\Pi_i$ on the list of $\Mod_\S$ representatives of pants decompositions of $\S$, so that Proposition \ref{prop:BdrySubwordIntBound} holds for the pair $(\gamma'', \Pi_i)$. This completes the proof of Proposition \ref{prop:BdrySubwordIntBound}.


\section{Length bound}
\label{sec:LengthBounds}
We also want to relate the length of $\gamma$ to the length of the word $w_\Pi(\gamma)$ for our nice pants decomposition $\Pi$.

Given $\gamma \in \G^c(L,K)$, we found a curve $\gamma' \in \Mod_\S \cdot \gamma$ and a pants decomposition $\Pi$ that satisfied the intersection number conditions in Proposition \ref{prop:BdrySubwordIntBound}. That is, if $w_\Pi(\gamma') = b_1s_1 \dots b_ns_n$, then $\sum |b_i| \lesssim K$ and $n \lesssim \sqrt K$. We now get a condition on $w_\Pi(\gamma')$ in terms of $L$.

\begin{lem}
\label{lem:BdrySubwordLKBounds}
Let $\gamma \in \G^c(L,K)$. Then there is a curve $\gamma'\in \Mod_\S \cdot \gamma$ and a pants decomposition $\Pi$ that satisfy the conditions of Proposition \ref{prop:BdrySubwordIntBound} and so that if $w_\Pi(\gamma') = b_1s_1 \dots b_n s_n$, then
 \[
  |w_\Pi(\gamma')| \leq c_X L + c_\S \sqrt K
 \]
 where $c_X$ depends only on the metric $X$ and $c_\S$ depends only on $\S$.
\end{lem}
Note that the constant $c_\S$ in this lemma is actually 18 times the constant $c_\S$ from Proposition \ref{prop:NicePants}.
\begin{proof}
First, we show that if $\gamma \in \G^c(L,K)$ for $L,K > 0$, and if $\Pi$ is a pants decomposition of $\S$ so that $i(\gamma, \Pi) \leq c_\S \sqrt K$, then
\[
 |w_\Pi(\gamma)| \leq \frac{18L}{l_X} + 18 c_\S \sqrt K
\]
where $l_X$ is the length of the systole in $X$. Note that this does not quite complete the proof, as there is no guarantee that the curve $\gamma' \in \Mod_\S \cdot \gamma$ for which Proposition \ref{prop:BdrySubwordIntBound} holds also has length at most $L$.

\begin{figure}[h!]
 \centering 
 \includegraphics{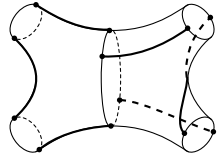}
 \caption{The hexagon decomposition $H_\perp$ by right-angled hexagons.}
 \label{fig:PerpHexagonDecomposition}
\end{figure}

The pants decomposition $\Pi$ cuts $\S$ into pair of pants. Further cut each pair of pants into \textit{right-angled} hexagons (Figure \ref{fig:PerpHexagonDecomposition}). Once again, these hexagons have boundary edges that lie on curves in $\Pi$, and seam edges that join curves in $\Pi$ together. Let $H_\perp$ be the set of right-angled hexagons we obtain. 

Then $H_\perp$ cuts $\gamma$ into segments, which are maximal subarcs of $\gamma$ lying entirely in a single hexagon. Let $n$ be the number of segments of $\gamma$ with respect to $H_\perp$. Then we claim that
\[
 |w_\Pi(\gamma)| \leq 9n
\]
In fact, let $H$ be the hexagon decomposition of $\S$ used to define $w_\Pi(\gamma)$. The hexagons in $H$ are not right-angled, because their seems are forced to match up across curves in $\Pi$. But their seam edges are chosen to be as short as possible. So if $m$ is the number of segments of $\gamma$ with respect to $H$, then $m \leq 3n$. Let $p(\gamma)$ be the closed curve formed by concatenating the edges in $w_\Pi(\gamma)$. There is a homotopy between $\gamma$ and $p(\gamma)$ that sends each segment of $\gamma$ with respect to $H$ to a subarc of at most three edges in $p(\gamma)$. Therefore, $|w_\Pi(\gamma)| \leq 3m$, and so $|w_\Pi(\gamma)| \leq 9n$.

This means we need to bound the number $n$ in terms of $l(\gamma)$.

\begin{figure}[h!]
 \centering 
 \includegraphics{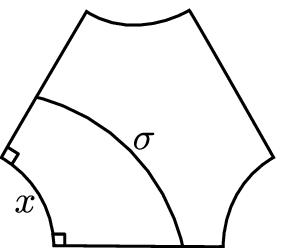}
 \caption{A segment $\sigma$ connects two seam edges incident to a boundary edge $x$.}
 \label{fig:ConnectSeams}
\end{figure}

Let $\sigma$ be a segment of $\gamma$ with respect to $H_\perp$. Suppose $\sigma$ joins two seam edges of some hexagon $h \in H_\perp$ (Figure \ref{fig:ConnectSeams}). Those two seam edge are connected by a boundary edge $x$ of $h$. Since the seam edges meet $x$ at right angles, we have that $l(\sigma) \geq l(x)$. But $l(x)$ is half the length of a curve in $\Pi$, and thus at least half the systole length of $X$. Therefore,
\[
 l(\sigma) \geq \frac 12 l_X
\]
Since the total length of $\gamma$ is at most $L$, this means there are at most $\frac{2L}{l_X}$ segments that join seam edges.

Now consider the set of segments that have at least one endpoint on the boundary edge of some hexagon $h \in H_\perp$. If $\sigma$ is such a segment, then it could join a seam edge to an adjacent boundary edge. Thus, $l(\sigma)$ can be arbitrarily small. However, each intersection between $\gamma$ and $\Pi$ corresponds to exactly two such segments. As $i(\gamma, \Pi) \leq c_\S \sqrt K$, the number of segments that touch a boundary edge is at most $2 c_\S \sqrt K$. (Note that we are overcounting segments that join two boundary edges together by a factor of 2.)

Therefore, the total number $n$ of segments of $\gamma$ with respect to $H_\perp$ is bounded above by
\[
 n \leq \frac{2L}{l_X} + 2 c_\S \sqrt K
\]
So we get
\[
 |w_\Pi(\gamma)| \leq \frac{18L}{l_X} + 18 c_\S \sqrt K
\]

 Now let $f$ be a composition of Dehn twists about curves in $\Pi$ so that the total length of bridging boundary subwords of $w_\Pi(f \cdot \gamma)$ is as small as possible. So by Lemma \ref{lem:BridgingBdrySubwordIntBound}, the total length of bridging boundary subwords in the pair $(f \cdot \gamma, \Pi)$ must be bounded by a multiple of $K$. As shown in the proof of Claim \ref{cla:ApplyDehnTwists}, applying Dehn twists to $\gamma$ changes only the bridging boundary subwords. Therefore, the pair $(f\cdot \gamma, \Pi)$ also satisfy the conditions of Lemma \ref{lem:InteriorBdrySubwordIntBound}. Moreover,
 \[
  |w_\Pi(f \cdot \gamma)| \leq |w_\Pi(\gamma)|
 \]
 
Thus, there is a $\gamma' = f \cdot \gamma \in \Mod_\S \cdot \gamma$ and a pants decomposition $\Pi$ that satisfy both the conditions of Proposition \ref{prop:BdrySubwordIntBound} and of Lemma \ref{lem:BdrySubwordLKBounds}. As explained in Section \ref{sec:PantsNotOnTheList}, this means that there is a curve $\gamma'' \in \Mod_\S \cdot \gamma$ and a pants decomposition $\Pi_i$ on the representative list of pants decompositions, so that the pair $(\gamma'', \Pi_i)$ satisfy the conditions of Lemma \ref{lem:BdrySubwordLKBounds}. 

\end{proof}

\begin{cor}
 If $l(\gamma) \leq L$, and $\Pi$ is a pants decomposition so that $i(\gamma, \Pi) \leq c_\S \sqrt{i(\gamma, \gamma)}$, then
 \[
  |w_\Pi(\gamma)| \lesssim L
 \]
 where the constant depends on the metric $X$
\end{cor}
This follows from the previous lemma, and the fact that $i(\gamma, \gamma) \leq \kappa L^2$, where the constant $\kappa$, which depends only on the metric $X$, can be found in, for example, \cite{Basmajian13}.

\begin{rem}
 One can in fact show that for any geodesic $\gamma$ and pants decomposition $\Pi$,
 \[
  \frac{1}{27} l_{min} |w_\Pi(\gamma)| \leq l(\gamma) \leq 2l_{max}|w_\Pi(\gamma)|
 \]
 where $l_{min}$ and $l_{max}$ are the longest and shortest edge lengths in a right-angled hexagon decomposition of $\Pi$. We do not do this here, since it is not needed for the proof of the main theorem.
\end{rem}


\section{Proof of Theorem \ref{thm:Main}}
\label{sec:CountingWords}
We are now ready to prove Theorem \ref{thm:Main}. Recall from Section \ref{sec:MainIdeaOfProof} that $\{\Pi_1, \dots, \Pi_l\}$ was our representative list of pants decompositions of $\S$, containing one pants decomposition from each $\Mod_\S$ orbit. Then for any $L,K > 0$, we defined $\W_j(L,K)$ to be the set of all cyclic words $w_{\Pi_i}(\gamma)$ so that $\gamma$ is non-simple, and if $w_{\Pi_i}(\gamma) \in \W_j(L,K)$, then 
\begin{enumerate}
 \item \label{item:TotalSubwordLengthVsIntersection} If $w_{\Pi_i}(\gamma) = b_1s_1 \dots b_ns_n$, then
 \[
 \sum_{i=1}^n |b_i| \leq d_\S K
 \]
 where $d_\S$ depends only on $\S$.
 \item \label{item:TotalSubwordNumberVsIntersection} Furthermore,
 \[
  n \leq d_\S \sqrt K
 \]
 where $d_\S$ depends only on $\S$.
 \item \label{item:TotalSubwordLengthVsLength} And lastly,
 \[
  |w_\Pi(\gamma)| \leq d_X L + d_\S \sqrt K
 \]
 where $d_X$ depends only on $X$ and $d_\S$ depends only on $\S$.
\end{enumerate}
Then for any $\gamma \in \G^c(L,K)$, Proposition \ref{prop:BdrySubwordIntBound} and Lemma \ref{lem:BdrySubwordLKBounds} imply that there is a geodesic $\gamma' \in \Mod_\S \cdot \gamma$ and pants decomposition $\Pi_i$ in the representative list, so that $w_{\Pi_i}(\gamma') \in \W_j(L,K)$.

By choosing such a $\gamma'$ for each $\gamma \in \G^c(L,K)$, we can define a map 
\[
  \Oo(L,K) \rightarrow \bigcup_{i=1}^l W_i(L,K)
\]
In fact, this map is one-to-one. To see this, note that the letters in $w_\Pi(\gamma)$ can be concatenated into a curve freely homotopic to $\gamma$. So $w_\Pi(\gamma) = w_\Pi(\gamma')$ implies $\gamma = \gamma'$. In particular, two distinct $\Mod_\S$ orbits cannot be sent to the same cyclic word.

So to bound $\#\Oo(L,K)$, we first get a slightly obscure upper bound on the size of $\W_j(L,K)$ for each $j$. We then simplify the upper bound, and sum over all pants decompositions in the representative list to get Theorem \ref{thm:Main}.

\subsection{Bound on the Size of $\W_j(L,K)$}
The following lemma gives a general form for an upper bound on $\#\W_j(L,K)$. 

\begin{lem}
\label{lem:SizeWLK}
 Let $\S$ be a compact genus $g$ surface with $b$ geodesic boundary components. Fix $L,K >0$. Then the size of $\W_j(L,K)$ is bounded above as follows:
 \[
  \# \W_j(L,K) <  (8m)^N \cdot {M+N \choose N}
 \]
for 
\[
M = \min \{\lfloor d_XL + d_\S \sqrt K \rfloor, d_\S K\} \mbox{ and } N = \lfloor d_\S \sqrt K \rfloor
\]
where $m = 3g-3+2b$, $d_X$ depends only on the metric $X$ and $d_\S$ depends only on $\S$.
\end{lem}

\begin{proof}
Suppose $w = b_1 s_1 \dots b_n s_n \in \W_j(L,K)$. Note that Condition \ref{item:TotalSubwordLengthVsLength} actually implies that $\sum |b_i| \leq d_X L + d_\S \sqrt K$. So, in fact, Conditions \ref{item:TotalSubwordLengthVsIntersection} - \ref{item:TotalSubwordLengthVsLength} imply that
\begin{equation}
\label{eq:SequenceBounds}
 \sum_{i=1}^n |b_i| \leq M \text{ and } n \leq N
\end{equation}
where 
\[
 M = \{\lfloor d_XL + d_\S \sqrt K \rfloor, d_\S K\} \text{ and } N = \lfloor d_\S \sqrt K \rfloor
\]

To get a bound on $\# \W_j(L,K)$, we will bound the number of cyclic words $b_1 s_1 \dots b_n s_n$ satisfying the inequalities in (\ref{eq:SequenceBounds}). In fact, given a sequence $(b_1, b_2, \dots, b_n)$ of boundary subwords, there is at most one sequence $(s_1, \dots, s_n)$ so that $b_1 s_1 \dots b_n s_n$ is a cyclic word in $\W_{\Pi_j}$. So we just need to bound the number of sequences $(b_1, \dots, b_n)$ of boundary subwords with the above properties. 

Note that the sequence $(b_1, \dots, b_n)$ may have empty boundary subwords that just encode the vertex where $s_{i-1}$ and $s_i$ meet. Furthermore, the sequences are not cyclic, so more than one sequence corresponds to the same cyclic word. Since we only want an upper bound, we ignore this fact.

Because words in $\W_i(L,K)$ do not back-track, each boundary subword $b_i$ is uniquely determined by its initial boundary edge $x_i$ and its length $|b_i|$. In other words, a pair of sequences $(x_1, \dots, x_n)$ of boundary edges and $(l_1, \dots, l_n)$ of non-negative integers determines a sequence $(b_1, \dots, b_n)$, where $b_i$ has initial boundary edge $x_i$ and length $l_i = |b_i|$ for each $i$. (If $l_i = 0$ for some $i$, then $b_i$ is just the start point of the oriented edge $x_i$.)

Fix an $n < N$. As there are $4m$ oriented boundary edges, the number of length $n$ sequences $(x_1, \dots, x_n)$ is $(4m)^n$. As $n \leq N$, each sequence $(l_1, \dots, l_n)$ of lengths corresponds to at most $(4m)^N$ sequences $(b_1, \dots, b_n)$. So we can just count the number of sequences $(l_1, \dots, l_n)$ so that 
\[
  \sum_{i=1}^n l_i \leq M  \mbox{ and }  n \leq N .
\]

We need to count the number of ways to write all numbers smaller than $M$ as ordered sums of at most $N$ non-negative integers. (Note that this problem makes sense precisely because $M > N \geq 1$.) This is bounded above by $N{M + N \choose N}$. To see this, suppose we put down $M+N$ rocks in a row and choose $N$ of them to pick up. There are $M+N \choose N$ ways to do this. Then we are left with $N+1$ groups of rocks, some of which may be empty. So we get a way to write $M$ as an ordered sum of at $N+1$ non-negative integers. If we choose a number $n \leq N$, we can choose the first $n$ terms of this sum, $l_1 + \dots + l_n$ with $l_i \geq 0, \forall i$. There are $N$ ways to choose $n$. Setting $M' = l_1 + \dots + l_n$, we see that this is a way to write a number $M' \leq M$ as the sum of $n \leq N$ non-negative integers. 

Furthermore, for each number $M' \leq M$ and each way to write it as $M' = l_1 + \dots + l_n$, there is a way of choosing $N$ rocks out of a row of $M+N$ so that the first $n$ groups of rocks correspond to exactly this sum. For example, we can choose rocks numbers $l_1 + 1$, $l_2 + 2$, and so on, through $l_n +n$, and then choose rocks numbered $l_n + n+1, l_n+n + 2, \dots, l_n + N$, to get the sequence $l_1, l_2, \dots, l_n$. 

Thus, the number of sequences $(l_1, \dots, l_n)$ is at most
\[
 N{M+N \choose N}
\]
Therefore, the number of sequences $(b_1, \dots, b_n)$ is at most $ N (4m)^N \cdot {M+N \choose N}$. As $N \leq 2^N$ for all $N \geq 1$, we get that the number of sequences $(b_1, \dots, b_n)$ satisfying the two inequalities in (\ref{eq:SequenceBounds}) is at most
\[
 (8m)^N \cdot {M+N \choose N}
\]
\end{proof}

To get nice upper bounds on the size of $\Oo(L,K)$, we thus want to bound binomial coefficients of the form ${a+b \choose b}$.

\begin{lem}
\label{lem:ABchooseB}
 Suppose $a,b \in \N$ with $a \geq b \geq 1$. Then
 \[
  {a+b \choose b} \leq  e^{b \log \left ( 2e \frac ab \right )}
\]
\end{lem}
\begin{proof}
We get this formula via the following computation.
 \begin{align*}
 {a+b \choose b} &= \frac{1 \cdots (a +b)}{(1 \cdots b)(1 \cdots a)}\\
 & = \frac{(a+1) \cdots (a+b)}{1 \cdots b}\\
 & \leq \frac{2a}{1} \cdot \frac{2a}{2} \cdots \frac{2a}{b}\\
 & = \frac{(2a)^b}{b!}
\intertext{Stirling's formula gives us, in particular, that $b! \geq \sqrt{2 \pi b}\left (\frac be \right)^b$ for all $b \geq 1$. So,}
\frac{(2a)^b}{b!} & \leq \frac{(2a)^b}{\sqrt{2 \pi b}\left (\frac be \right)^b} \\
& \leq \frac{1}{\sqrt{2 \pi b}} e^{b \log \left ( 2e \frac ab \right )}
\end{align*}
Since $\frac{1}{\sqrt{2 \pi b}} < 1$, we get,
\[
  {a+b \choose b} \leq  e^{b \log \left ( 2e \frac ab \right )}
\]
\end{proof}


\subsection{Proof of Theorem \ref{thm:Main}.}
 The proof of Theorem \ref{thm:Main} is just an application of Lemma \ref{lem:ABchooseB} to the upper bound we found on $\#\W_j(L,K)$ in Lemma \ref{lem:SizeWLK}.

\begin{proof}[Proof of Theorem \ref{thm:Main}]

We have that 
\[
 \#\Oo(L,K) \leq \sum \# \W_j(L,K) + \#\Oo(L,0)
\]
where $\#\Oo(L,0) = 1 + \lfloor \frac g 2 \rfloor$. Thus, by Lemma \ref{lem:SizeWLK},
\[
 \# \Oo(L,K) \leq l \cdot (8m)^N \cdot {M+N \choose N} +  \#\Oo(L,0)
\]
where 
\[
M = \min \{\lfloor d_XL + d_\S \sqrt K \rfloor, d_\S K\} \mbox{ and } N = \lfloor d_\S \sqrt K \rfloor
\]

Applying Lemma \ref{lem:ABchooseB}, this gives us that
\begin{align*}
 \# \Oo(L,K) & \leq l \cdot (8m)^N \cdot e^{N \log \left ( 4e \frac MN \right )}  +  \#\Oo(L,0) \\
  & = l \cdot e^{N \log \left ( 4e \frac MN + 8m \right )}  +  \#\Oo(L,0) 
\end{align*}

We expand the exponent $N \log \left ( 2e \frac MN + 8m \right )$ in the cases where $M = \lfloor d_XL + d_\S \sqrt K \rfloor$ and $d_\S K$, respectively. Note that $\lfloor d_\S \sqrt K \rfloor \geq \frac 12 d_\S \sqrt K$, as $K \geq 1$.

First, suppose $M = \lfloor d_XL + d_\S \sqrt K \rfloor$. Then
\begin{align*}
 N \log \left ( 2e \frac MN + 8m \right )  & \leq d_\S \sqrt K \log \left ( 4e \frac{d_XL + d_\S \sqrt K + 8m }{d_\S \sqrt K} \right ) \\
 & = d_\S \sqrt K \log \left ( 4e \frac{d_XL}{d_\S \sqrt K} + 3e \right )
\end{align*}
as $K \geq 1$ and we can assume $\frac{8m}{d_\S} \leq 1$. So we can define a new constant $\overline{d_X}$ depending only on $X$ so that in this case,
\[
 \# \Oo(L,K) \leq e^{d_\S \sqrt K \log \left ( \overline{d_X} \frac{L}{\sqrt K} + \overline{d_X} \right )}
\]
Since both $l$ and $\#\Oo(L,K)$ are constants depending only on $\S$, and we increase $d_\S$ if necessary to incorporate them into the exponent.

Now suppose that $M = d_\S K$. Then
\begin{align*}
 N \log \left ( 2e \frac MN + 8m \right )  & \leq d_\S \sqrt K \log \left ( 2e \frac{d_\S K + 8m }{d_\S \sqrt K} \right ) \\
 & = d_\S \sqrt K \log (2e(d_\S \sqrt K) + e)
\end{align*}
as $K \geq 1$ and we can assume $\frac{8m}{d_\S} \leq 1$. So we can define a new constant $\overline{d_\S}$ depending only on $\S$ so that in this case,
\[
 \# \Oo(L,K) \leq e^{\overline{d_\S} \sqrt K \log \overline{d_\S} \sqrt K}
\]
Since both $l$ and $\#\Oo(L,K)$ are constants depending only on $\S$, we can increase $\overline{d_\S}$ if necessary to incorporate them into the exponent.

As $\#\Oo(L,K)$ is bounded above by the smaller of the two bounds given here, we have the theorem.
\end{proof}

\bibliographystyle{alpha}
\bibliography{recount}
\end{document}